\documentclass[11pt,amssymb]{amsart}
\usepackage{amssymb}
\usepackage[mathscr]{eucal}
\usepackage[all,cmtip]{xy}
\newcommand{\PP}{\mathbb{P}}

\newcommand{\Z}{{\mathbb Z}}

\newcommand{\Q}{{\mathbb Q}}

\numberwithin{equation}{section}
\newcommand{\Br}{\mathrm{Br}}

\newcommand{\cO}{\mathscr{O}}

\newcommand{\Ga}{\mathrm{Gal}}
\newtheorem{thm}{Theorem}[section]
\newtheorem{lemma}[thm]{Lemma}
\newtheorem{prop}[thm]{Proposition}
\newtheorem{cor}[thm]{Corollary}

\newcommand{\cG}{\mathcal{G}}

\newcommand{\gen}{\mathbf{gen}}
\newcommand{\Hom}{\mathrm{Hom}}

\newcommand{\Div}{\mathrm{Div}}
\newcommand{\Pic}{\mathrm{Pic}}
\newcommand{\Spec}{\mathrm{Spec}}

\newcommand{\he}{H_{\mathrm{\acute{e}t}}}
\newcommand{\Xet}{X_{{\rm \acute{e}t}}}
\newcommand{\Yet}{Y_{{\rm \acute{e}t}}}
\begin{document}

\title[Genus of a division algebra]{On the size of the genus of a division algebra}

\author[V.~Chernousov]{Vladimir I. Chernousov}
\author[A.~Rapinchuk]{Andrei S. Rapinchuk}
\author[I.~Rapinchuk]{Igor A. Rapinchuk}

\begin{abstract}
Let $D$ be a central division algebra of degree $n$ over a field $K$. One defines the genus $\gen(D)$ as the set
of classes $[D'] \in \Br(K)$ in the Brauer group of $K$ represented by central division algebras $D'$ of degree $n$
over $K$ having the same maximal subfields as $D$. We prove that if the field $K$ is finitely generated and $n$ is
prime to its characteristic then $\gen(D)$ is finite, and give explicit estimations of its size in certain situations.
\end{abstract}

\address{Department of Mathematics, University of Alberta, Edmonton, Alberta T6G 2G1, Canada}

\email{vladimir@ualberta.ca}

\address{Department of Mathematics, University of Virginia,
Charlottesville, VA 22904-4137, USA}

\email{asr3x@virginia.edu}

\address{Department of Mathematics, Harvard University, Cambridge, MA, 02138 USA}

\email{rapinch@math.harvard.edu}

\maketitle

\hfill {\it To V.P.~Platonov on his 75th birthday}

\section{Introduction}\label{S:Intro}

Let $K$ be a field. For a finite-dimensional central simple algebra $A$ over $K$, we let $[A] \in \Br(K)$ denote the corresponding class in the Brauer group.
Given a central division algebra $D$ of degree $n$ over $K$, one defines the {\it genus} $\gen(D)$ as the set of  classes $[D'] \in \Br(K)$ represented by
division algebras $D'$ having the same maximal subfields\footnotemark as $D$ (cf. \cite{CRR1}, \cite{CRR2}; a survey can be found in \cite{CRR3}).
The goal of this paper is to give a detailed and effective proof of the following
result announced in \cite{CRR1}. \footnotetext{This means that $D'$ also has degree $n$, and a field extension $P/K$ of degree $n$ admits a $K$-embedding $P
\hookrightarrow D$ if and only if it admits a $K$-embedding $P \hookrightarrow D'$.}

\medskip

\noindent {\bf Theorem 1.} {\it Let $K$ be a finitely generated field, and let $n > 1$ be an integer prime to $\mathrm{char}\, K$. Then for any central division $K$-algebra
$D$ of degree $n$, the genus $\gen(D)$ is finite.}

\medskip

In \cite{CRR2}, the proof of the finiteness of $\gen(D)$ was reduced to the finiteness of a certain unramified Brauer group. In order to state the precise result, we need to recall some relevant definitions. Let $K$ be a field, $n > 1$ be an integer, and let $v$ be a discrete valuation of $K$. If the residue field $K^{(v)}$ is either perfect or of characteristic prime to $n$, there exists a {\it residue map} defined on the
$n$-torsion subgroup of the Brauer group:
$$
\rho_v \colon {}_n\Br(K) \to \Hom (\cG^{(v)}, \Z/ n \Z),
$$
where $\cG^{(v)}$ denotes the absolute Galois group of $K^{(v)}$ (cf. \cite[\S 10]{Salt} or
\cite[Ch.II, Appendix]{Serre}). We say that a class $[A] \in {}_n\Br(K)$ (or a finite-dimensional central simple $K$-algebra $A$ representing this class) is {\it unramified} at $v$ if $[A] \in \ker \rho_v$, and {\it ramified} otherwise. Furthermore, given a set $V$ of discrete valuations of $K$ such that the residue maps $\rho_v$ exist for all $v \in V$, the ($n$-{\it torsion of the}) {\it unramified Brauer group with respect to $V$} is defined as
$$
{}_n\Br(K)_V = \bigcap_{v \in V} \ker \rho_v.
$$
Assume now that a given finitely generated field $K$ is equipped with a set $V$ of discrete valuations that satisfies the following two conditions (the second of which
depends on $n$):

\medskip

\noindent (A) For any $a \in K^{\times}$, the set $V(a) := \{ v \in V \: \vert \: v(a) \neq 0 \}$ is finite;

\smallskip

\noindent (B) For any $v \in V$, the characteristic $\mathrm{char}\: K^{(v)}$ is prime to $n$ (then, in particular, $\rho_v$ is defined).

\medskip

\noindent It was shown in \cite{CRR2} that if $D$ is a central division $K$-algebra of degree $n$, then (i) the set $R(D)$ of those $v \in V$ where $D$ ramifies is finite;
(ii) if ${}_n\Br(K)_V$ is finite, then so is $\gen(D)$ and
\begin{equation}\label{E:genus1}
\vert \gen(D) \vert \leqslant \varphi(n)^r \cdot \vert {}_n\Br(K)_V \vert \ \  \text{with} \ \  r = \vert R(D) \vert,
\end{equation}
where $\varphi$ is the Euler function. Thus, Theorem 1 is a direct consequence of the following.

\medskip

\noindent {\bf Theorem 2.} {\it Let $K$ be a finitely generated field, and $n > 1$ be an integer prime to $\mathrm{char}\: K$. There exists a set $V$ of discrete valuations of $K$ that satisfies conditions {\rm (A)} and {\rm (B)} and for which the unramified Brauer group ${}_n\Br(K)_V$ is finite.}

%

\medskip

We will now outline the construction of a required set $V$ in the case where $K$ is the function field $k(C)$ of a smooth projective geometrically
irreducible curve $C$ over a number field $k$. (This construction, on the one hand, generalizes the explicit construction we gave in \cite[\S4]{CRR2} in the situation
where $n = 2$ and $C$ is an elliptic curve over $k$ with $k$-rational 2-torsion, and on the other hand extends relatively easily  to the general case - see below.) Let $V_0$
be the set of (nontrivial) discrete valuations of $K$ that are trivial on $k$ - these correspond to the closed points of $C$ and will be referred to as {\it geometric places}
in the sequel. Next, we choose a sufficiently large finite subset $S$ of the set $V^k$ of all valuations of $k$ that contains all the archimedean places and those nonarchimedean ones  which either divide $n$ or where $C$ does not have good reduction (see \S\ref{S:Intro}.2 below and \S\ref{S:PfT2}.2 for more details). Each $v \in T := V^k \setminus S$ has a canonical extension $\dot{v}$ to $K$, and we set $\dot{T} = \{ \dot{v} \: \vert \:  v \in  T \}$. Then $V := V_0 \cup \dot{T}$ is as required. In fact, not only do we prove that ${}_n\Br(K)_V$ is finite, but also give a rather explicit bound on its order. We describe the general method for obtaining such bounds in \S\ref{S:PfT2}.2, and here show only one sample statement.  Let $J$ be the Jacobian of $C$.

\medskip

\noindent {\bf Theorem 3.} (cf. Theorem \ref{T:ExplEst} for a more precise result) {\it Let $C$ be as above. Assume that $C(k) \neq \emptyset$ and
$J$ has $k$-rational $n$-torsion. Then the
order of ${}_n\Br(K)_V$ divides $n^{(2g+1)\vert S \vert} \cdot h_k(S)^{2g}$ where $g$ is the genus of $C$ and $h_k(S)$ is the class number of the ring $\mathcal{O}_k(S)$
of $S$-integers in $k$.}

\medskip

We will now indicate the strategy of the proof of Theorem 2 in the general case assuming that $\mathrm{char}\: K =0$ (the general case requires only minor technical - mostly
notaional - adjustments, see \S \ref{S:BrExactSeq}). We begin by presenting a given finitely
generated field $K$ as  a function field $k(C)$  where $k$ is a purely transcendental extension of some number field $P$ (= the algebraic closure of $\Q$ in $K$) and $C$ is a smooth projective geometrically irreducible curve over $k$. We again let $V_0$ denote the set of geometric places of $K$, and take $V = V_0 \cup \widetilde{V}_1$ where $\widetilde{V}_1$ consists of extensions to $K$ of the discrete valuations of $k$ from a certain specially constructed set of places $V_1$. For $v \in V_0$, one can define the residue map on the entire Brauer group
$$
\rho_v \colon \Br(K) \longrightarrow H^1(\mathcal{G}^{(v)} , \Q/\Z).
$$
The corresponding unramified Brauer group $\Br(K)_{V_0} = \bigcap_{v \in V_0} \ker \rho_v$ is by tradition denoted $\Br(K)_{\mathrm{ur}}$, and  is known to coincide with the geometric Brauer group $\Br(C)$ defined either in terms of Azumaya algebras or in terms of \'etale cohomology (see \cite[\S 6.4]{GiSz} and \cite{Lich}). The proof of Theorem 2 is based on an analysis of the well-known exact sequence for this group:
\begin{equation}\label{E:Out1}
\Br(k) \stackrel{\iota_k}{\longrightarrow} \Br(k(C))_{\rm ur} \stackrel{\omega_k}{\longrightarrow} H^1 (k,J)/\Phi (C,k),
\end{equation}
where $\iota_k$ is the natural map, $J$ is the Jacobian of $C$, and $\Phi (C,k)$ is a certain finite cyclic subgroup of $H^1(k,J)$ (see \S \ref{S:BrExactSeq}). More precisely, for  $V$ as above, we have the inclusion ${}_n\Br(K)_V \subset {}_n\Br(K)_{\mathrm{ur}}$, so proving the finiteness of ${}_n\Br(K)_V$ (and estimating its order) reduces to proving the finiteness of
$$
M := \iota_k^{-1}({}_n\Br(K)_V) \ \ \  {\rm and}  \ \ \  N :=\omega_k({}_n\Br(K)_V)
$$
(and estimating their respective orders). To establish the finiteness of $M$ (see Theorem \ref{T:FinGenF1}), one shows that $M$ contains the unramified Brauer group ${}_n\Br(k)_{V_1}$ as a subgroup of finite index. The finiteness of ${}_n\Br(k)_{V_1}$ is then derived from the construction of $V_1$; the argument here relies on the fact that $k$ is a {\it purely transcendental} extension of the number field $P$ and makes use of  the Faddeev exact sequence and the Albert-Hasse-Brauer-Noether theorem (cf. \S \ref{S:Constants}). In order to prove the finiteness of $N$, we  reduce the problem
to proving  the finiteness of the unramified cohomology group $H^1(k , {}_nJ)_{V_1}$, where ${}_nJ$ stands for the $n$-torsion of the Jacobian $J$ of the curve $C$. The latter is proved by an argument that imitates  the proof of the Weak Mordell-Weil Theorem (cf. \cite[ch. VIII, \S\S1-2]{Silv}).

After completing the proof of Theorem 2 as outlined above, we learned from J.-L.~Colliot-Th\'el\`ene that one can also derive it from results on \'etale cohomology. For this, one
realizes $K$ as the function field of a smooth arithmetic scheme $X$ on which $n$ is invertible, and lets $V$ to be the set of discrete valuations of $K$ associated with prime divisors on $X$. The proof of the finiteness of ${}_n\Br(K)_V$ is then derived from Deligne's finiteness theorem for the \'etale cohomology of constructible sheaves (see ``Th\'eor\`ems de finitude'' in \cite{Del}) and Gabber's purity theorem \cite{Fuj} - see the Appendix for the details. While this proof does not lead to explicit estimations on the size of the genus and/or the unramified Brauer group, it has the advantage of allowing more flexibility in the choice of $V$: for example, $X$ can be replaced by an open subscheme, enabling us to remove from $V$ any finite subset and still preserve  the finiteness of the corresponding unramified Brauer group. (In fact, this can also be accomplished using our argument.)

It should be noted that the assumption of finite generation of $K$ is essential for the finiteness of the genus - see \cite{Meyer} and \cite{Tikh} for a construction of division
algebras with infinite genus in the general situation.

\medskip

\noindent {\bf \ref{S:Intro}.2. Notations and conventions.} Given a discrete valuation $v$ of a field $k$, we let $\mathcal{O}_{k, v}$, or simply $\mathcal{O}_v$,  denote the valuation ring, and set  $k^{(v)}$ and $k_v$ to be the corresponding residue field and the completion of $k$ at $v$, respectively. Let $C$ be a smooth projective geometrically irreducible curve over  $k$. We say that $C$ has \emph{good reduction} at $v$ if there exists a smooth curve $\mathscr{C}$ over $\mathcal{O}_v$ with the
generic fiber $\mathscr{C} \times_{\mathcal{O}_v} k$ isomorphic to $C$ such that the special fiber (reduction) $\underline{C}^{(v)} := \mathscr{C} \times_{\mathcal{O}_v}
k^{(v)}$ is (smooth and) geometrically irreducible. Any such choice of the $\mathcal{O}_v$-model $\mathscr{C}$  determines a unique unramified extension $\dot{v}$ of $v$ to $k(C)$ (corresponding to the special fiber) with  residue field $k(C)^{(\dot{v})} = k^{(v)}(\underline{C}^{(v)})$.

If $k$ is a number field, we will use $V^k$ to denote set of all places  of $k$, and let $V^k_{\infty}$ denote the subset of  archimedean places.

For an integer $n > 1$ we let $\mu_n$ denote the group of $n$th roots of unity, and for an abelian group $A$, we let ${}_nA$ denote the $n$-torsion subgroup of $A$.
%

\vskip5mm


\section{An exact sequence for the Brauer group of a curve}\label{S:BrExactSeq}

In this section, we review the construction of an exact sequence generalizing (\ref{E:Out1}) to  fields of
arbitrary characteristic mainly in order to describe our set-up carefully and introduce the necessary notations.
For a slightly different approach, we refer the reader to Lichtenbaum's paper \cite{Lich}, where one can also find a derivation of (\ref{E:Out1}) from the Hochschild-Serre spectral sequence in \'etale cohomology, as well as a proof of the fact that $\Br(k(C))_{\rm ur}$ coincides with Grothendieck's geometric Brauer group of the curve $C$ (see also \cite[Ch. III]{GaMeSe}).

Throughout this section, we will use the following notations. Let $k$ be an arbitrary field. Fix an algebraic closure $k^{\mathrm{alg}}$ of $k$,  let
$\bar{k} = k^{\mathrm{sep}}$ be the separable closure of $k$ in $k^{\mathrm{alg}}$ and set $\Gamma = \Ga(\bar{k}/k)$ to be the absolute Galois group $k$.
Now, let $C$ be a smooth projective geometrically irreducible curve defined over $k$ with function
field $K = k(C)$ and denote by $J = J(C)$ the Jacobian of $C$. We will construct an exact sequence of the form
\begin{equation}\label{E:Out1a}
\Br(k) \stackrel{\iota_k}{\longrightarrow} \Br'(k(C))_{\mathrm{ur}} \stackrel{\omega_k}{\longrightarrow} H^1(k , J)/\Phi(C , k)
\end{equation}
where $\Br'(k(C))_{\mathrm{ur}}$ is the unramified part of
$$
\Br'(k(C)) := \Br(\bar{k}(C)/k(C)).
$$
We recall that as follows from Tsen's theorem (cf. \cite[Theorem 6.2.8]{GiSz}),  $\Br(k^{\mathrm{alg}}(C)) = 0$. So,
$\Br'(k(C))$ contains ${}_n\Br(k(C))$ for any $n$ prime to $\mathrm{char}\: k$ making (\ref{E:Out1a}) applicable in our set-up. (Obviously,
$\Br'(k(C))$ also contains $\Br(k)$.) We have
$$
\Br'(k(C)) = H^2(\Gamma , \bar{k}(C)^{\times}).
$$
To make use of this description, we let $\bar{C} = C \times_k \bar{k}$ (so that $\bar{k}(C) = \bar{k}(\bar{C}) = K\bar{k}$), and
consider the following standard exact sequences of $\Gamma$-modules
\begin{equation}\label{E:Out2a}
0 \to \bar{k}^{\times} \longrightarrow \bar{k}(C)^{\times}
\longrightarrow \mathrm{P}(\bar{C}) \to 0
\end{equation}
and
\begin{equation}\label{E:Out3a}
0 \to \Div^0(\bar{C}) \longrightarrow \Div(\bar{C})
\stackrel{\delta}{\longrightarrow} \Z \to 0
\end{equation}
where $\Div(\bar{C})$ is the group of divisors of $\bar{C}$, $\delta$ is the degree map (recall that $C(\bar{k}) \neq \emptyset$ making $\delta$ surjective), and $\Div^0(\bar{C})$
and $\mathrm{P}(\bar{C})$ are the subgroups of $\Div(\bar{C})$ of degree zero and principal divisors, respectively. Then (\ref{E:Out2a}) gives rise to the following exact sequence
\begin{equation}\label{E:Out4a}
\Br(k) \stackrel{\iota_k}{\longrightarrow} \Br'(K)
\stackrel{\alpha}{\longrightarrow} H^2(\Gamma ,
\mathrm{P}(\bar{C})),
\end{equation}
where $\iota_k$ is the natural base change map.

Now, the inclusions $\mathrm{P}(\bar{C}) \hookrightarrow \Div^0(\bar{C}) \hookrightarrow \Div(\bar{C})$ induce homomorphisms
$$
H^2(\Gamma , \mathrm{P}(\bar{C})) \stackrel{\beta^0}{\longrightarrow} H^2(\Gamma , \Div^0(\bar{C})) \stackrel{\varepsilon}{\longrightarrow} H^2(\Gamma , \Div(\bar{C}))
$$
and we set $\beta = \varepsilon \circ \beta^0$. It follows from \cite[6.4]{GiSz}  that the unramified Brauer group $\Br'(K)_{\mathrm{ur}}$ coincides\footnote{We note that
for any closed point $P$ of C and the corresponding valuation $v_P$ of $K$, the associated residue map $\rho_{v_P}$ is defined on the entire group $\Br'(K)$ because every
class in the latter splits over the maximal unramified extension of the completion $K_{v_P}$, cf. \cite[Ch. XII]{Serre-LF}.} with
the kernel of
\begin{equation}\label{E:def-rho-a}
\rho := \beta \circ \alpha \colon \Br'(K) \to H^2(\Gamma , \Div(\bar{C})).
\end{equation}
Since $H^1(\Gamma , \Z) = 0$, it follows from the cohomological sequence associated with (\ref{E:Out3a}) that $\varepsilon$ is injective. Thus,
\begin{equation}\label{E:Out775a}
\Br'(K)_{\mathrm{ur}}  = \ker \rho^0 \ \ \text{where} \ \ \rho^0 = \beta^0 \circ \alpha.
\end{equation}
On the other hand, the standard exact sequence defining the Picard group of $\bar{C}$
\begin{equation}\label{E:Picard-a}
0 \to \mathrm{P}(\bar{C}) \longrightarrow \Div^0(\bar{C})
\longrightarrow \Pic^0(\bar{C}) \to 0
\end{equation}
gives rise to the bottom exact sequence in the following diagram
\begin{equation}\label{E:Out5a}
\xymatrix{& & \Br(k) \ar[d]^{\iota_k} & \\ & & \Br'(K) \ar[d]^{\alpha} & \\ H^1(\Gamma, \Div^0(\bar{C})) \ar[r]^{\varphi} & H^1(\Gamma , \Pic^0(\bar{C})) \ar[r]^{\gamma} &
H^2(\Gamma , {\rm P}(\bar{C})) \ar[r]^{\beta^0} & H^2(\Gamma , \Div^0(\bar{C})).\\}
\end{equation}
It is well-known that $\Pic^0 (\bar{C})$ and $J(\bar{k})$ are isomorphic as $\Gamma$-modules (cf. \cite{MiJac}), and we will henceforth
routinely identify  $H^1(\Gamma , \Pic^0(\bar{C}))$ with $H^1(k , J)$. Furthermore, set $\Phi(C , k) = \mathrm{Im} \: \varphi$.
Using (\ref{E:Out775a}), we obtain from (\ref{E:Out5a}) that for every $a \in \Br(K)_{\mathrm{ur}}$ there exists $b \in H^1(k , J)$ such that $\gamma(b) =
\alpha(a)$, and the coset
$$
\bar{b} := b + \Phi(C , k) \in H^1(k , J)/\Phi(C , k)
$$
is well-defined. Then the correspondence $a \mapsto \bar{b}$ defines the required homomorphism
$$
\omega_k \colon \Br'(K)_{\rm ur} \to H^1(k , J)/\Phi(C , k),
$$
and the exactness of the resulting sequence (\ref{E:Out1a}) follows from the exactness of (\ref{E:Out4a}).

\begin{lemma}\label{L:Out2}
$\Phi(C , k)$ is a finite cyclic group. It is trivial if $C(k) \neq
\emptyset$.
\end{lemma}
\begin{proof}
From (\ref{E:Out3a}), we obtain the exact sequence
$$
\Div(\bar{C})^{\Gamma} \stackrel{\delta}{\longrightarrow} \Z
\longrightarrow H^1(\Gamma , \Div^0(\bar{C})) \longrightarrow
H^1(\Gamma , \Div(\bar{C})) = 0,
$$
in which the vanishing of the term on the right is derived from the fact that
$\Div(\bar{C})$ is a permutation $\Gamma$-module.
So,
$$
H^1(\Gamma , \Div^0(\bar{C})) \simeq
\Z/\delta(\Div(\bar{C})^{\Gamma}),
$$
hence $\Phi(C , k)$ is a finite cyclic group. If $C(k) \neq
\emptyset$ then $\delta(\Div(\bar{C}))^{\Gamma} = \Z$, and $\Phi(C ,
k)$ is trivial.
\end{proof}

\vskip2mm

\noindent {\bf Remark 2.2.} (a) Let $\ell / k$ be a field extension and $C_{\ell} = C \times_k \ell$ be the base change of $C$ (viewed as a curve over $\ell$). It is well-known (see, e.g., \cite{MiJac}) that $J(C_{\ell}) = J(C)_{\ell}.$ It follows from our construction that  the restriction map $h \colon H^1(k , J) \to H^1(\ell , J_{\ell})$ satisfies
$h(\Phi(C , k)) \subset \Phi(C_{\ell} , \ell)$, hence gives rise to a map $h'$ in the following commutative diagram
$$
\xymatrix{\Br'(k(C))_{\rm ur} \ar[d]_g \ar[r]^{\omega_k \ \ \ \ \ } & H^1 (k, J)/\Phi (C,k) \ar[d]^{h'} \\ \Br'(\ell (C))_{\rm ur} \ar[r]^{\omega_{\ell} \ \ \ \ \ } & H^1 (\ell, J_{\ell})/ \Phi (C_{\ell}, \ell)}
$$
where $g$ is the natural map.

\vskip1mm

\noindent (b) Suppose that $C(k) \neq \emptyset$. Then
(\ref{E:Out2a}) admits a $k$-defined splitting which is constructed as
follows. Fix $p \in C(k)$, let $v_p$ be the corresponding discrete
valuation on $\bar{k}(C)$, and pick a uniformizer $\pi \in k(C)$ of $v_p$. 
Then
$$
\bar{k}(C)^{\times} \ni f \stackrel{\sigma}{\mapsto}
(f\pi^{-v_p(f)})(p) \in \bar{k}^{\times}
$$
is a required section. Equivalently, $\sigma$ can be described as
the composition of the natural embedding $\bar{k}(C) \hookrightarrow
\bar{k}(\!(\pi)\!)$ with the map $\bar{k}(\!(\pi)\!)^{\times} \to
\bar{k}^{\times}$ defined by
$$
a_m\pi^m + a_{m+1}\pi^{m+1} + \cdots \ \mapsto \ a_m
$$
(assuming that $a_m \neq 0$). The existence of $\sigma$ yields the
exactness of the following sequence extending (\ref{E:Out4a}):
$$
0 \to \Br(k) \stackrel{\iota_k}{\longrightarrow} \Br'(k(C))
\stackrel{\alpha}{\longrightarrow} H^2(\Gamma , \mathrm{P}(\bar{C}))
\to 0.
$$
According to Lemma \ref{L:Out2}, we have $\Phi(C , k) = 0$, so
(\ref{E:Out1a}) can be replaced with the exact sequence
$$
0 \to \Br(k) \stackrel{\iota_k}{\longrightarrow} \Br'(k(C))_{\rm ur}
\stackrel{\omega_k}{\longrightarrow} H^1(k , J) \to 0.
$$

\vskip2mm

As we already mentioned, if $n$ is prime to $\mathrm{char}\: k$ then ${}_n\Br(K) \subset \Br'(K)$. This
enables us to formulate the following statement which summarizes our approach to proving the finiteness of the
unramified Brauer group.

\addtocounter{thm}{1}

\begin{prop}\label{P:Out1}
Let $V$ be a set of discrete valuations of $K = k(C)$ containing
$V_0$, so that we have an inclusion ${}_n\mathrm{Br}(K)_V \subset {}_n\mathrm{Br}(K)_{\rm ur}$.
Assume that

\vskip2mm

\ {\rm (I)} \parbox[t]{16cm}{$\iota^{-1}_k({}_n\Br(K)_V)$
is finite,
and}

\vskip1mm

{\rm (II)} $\omega_k({}_n\Br(K)_V)$ is finite.

\vskip2mm

\noindent Then ${}_n\Br(K)_V$ is finite and its order divides $\vert \iota^{-1}_k({}_n\Br(K)_V)\vert \cdot \vert \omega_k({}_n\Br(K)_V) \vert$.
\end{prop}

\section{Ramification at the valuations of the base field}\label{S:Constants}

The results of this section will be used  to verify  item (I) of Proposition \ref{P:Out1} for an appropriate set of discrete valuations (see \S\ref{S:FinGenF}.2).
We examine the following two situations: simple extensions of transcendence degree one of an arbitrary field and purely transcendental extensions of global fields. The consideration
of the first of these situations  involves  rather restrictive assumptions (cf. Proposition \ref{P:RC1}), so to conclude the proof of (I) we show that \emph{any} finitely generated
field can be obtained as a combination of these two types of extensions in such a way  that all the additional required conditions are satisfied (see Proposition \ref{P:FGF1}).

\medskip


\noindent {\bf \ref{S:Constants}.1. Ramification and base change.} First, we recall the following well-known general fact. Let $k$ be a field equipped with a (nontrivial) discrete valuation $v$, and let $n > 1$ be an integer relatively prime to $\mathrm{char} \: k^{(v)}$. Furthermore, let $\ell/k$ be a field extension, $w$ be an extension of $v$ to $\ell$ and $\ell^{(w)}$ be the corresponding residue field (as usual, we will identify $k^{(v)}$ with a subfield of $\ell^{(w)}$). We let $(k^{(v)})^{\mathrm{sep}} \subset (\ell^{(w)})^{\mathrm{sep}}$ denote the corresponding separable closures, and let $e = [w(\ell^{\times}) : v(k^{\times})]$ be the ramification index (assumed to be finite).
\addtocounter{thm}{1}
\begin{lemma}\label{L:RC1}
{\rm (\cite[Theorem 10.4]{Salt})} The diagram
\begin{equation}\label{E:RC777}
\xymatrix{ {}_n\Br(k) \ar[rr]^(.30){\rho_v} \ar[d]_{{}_n\iota} & &  \mathrm{Hom}(\Ga((k^{(v)})^{\mathrm{sep}}/k^{(v)}) , \Z/n\Z) \ar[d]^{[e]} \\ {}_n\Br(\ell) \ar[rr]^(.30){\rho_w} & & \mathrm{Hom}(\Ga(({\ell}^{(v)})^{\mathrm{sep}}/{\ell}^{(v)}) , \Z/n\Z)}
\end{equation}
where $_{n}\iota \colon {}_n\Br(k) \to {}_n\Br(\ell)$ is the canonical map, $\rho_v$ and $\rho_w$ are the corresponding residue maps, and $[e]$
is the map induced by the restriction $\Ga((\ell^{(w)})^{\mathrm{sep}}/\ell^{(w)}) \to \Ga((k^{(v)})^{\mathrm{sep}}/k^{(v)})$ followed by multiplication
by the ramification index $e$, is commutative.
\end{lemma}

\medskip

\begin{cor}\label{C:RC1}
In the above notations, set
$$
r = [{}_n\iota^{-1}({}_n\Br(\ell)_{\{ w \}}) :
{}_n\Br(k)_{\{v\}}] \ \ \text{and} \ \ d = [\ell^{(w)} \cap
(k^{(v)})^{\mathrm{sep}} : k^{(v)}].
$$
If $e = 1$, then $r \leqslant d$; in particular, if $d$ is finite then so is $r$, and moreover $r = 1$ if $d = 1$.
\end{cor}
Indeed, if $e = 1$ then the kernel of $[e]$ in (\ref{E:RC777}) has order dividing $d$, and our claim follows from a simple diagram chase.

\medskip

\addtocounter{thm}{1}

\noindent {\bf \ref{S:Constants}.4. Geometric statement.} Let $C$ be a smooth projective geometrically irreducible curve over a field $k$, and let $v$
be a discrete valuation of $k$. Assume that $C$ has good reduction at $v$ (cf. \ref{S:Intro}.2), fix an $\mathcal{O}_v$-model $\mathscr{C}$, and let
$\dot{v}$ denote the canonical unramified extension of $v$ to the field of rational functions $k(C)$ determined by the closed fiber of $\mathscr{C}$. Recall that $k(C)^{\dot{v}} =
k^{(v)}(\underline{C}^{(v)})$. Since
$C$ has good reduction at $v$, the residue field $k^{(v)}$ is algebraically closed in $k(C)^{(\dot{v})}$, and
we see from Corollary \ref{C:RC1} that
$$
{}_n\iota^{-1}({}_n\Br(k(C))_{\{ \dot{v} \}}) = {}_n\Br(k)_{\{ v \}}.
$$
Thus, we obtain the following.
\begin{prop}\label{P:RC333}
Let $C$ be a smooth projective geometrically irreducible curve over a field $k$, and let $V$ be a set of discrete valuations of $k$ such that $C$ has good reduction at all $v \in V$. Set $$\dot{V} = \{ \dot{v} \: \vert \: v \in V \},$$ where $\dot{v}$ is the canonical extension of $v$ to $k(C)$. Then for the natural map ${}_n\iota \colon {}_n\Br(k) \to {}_n\Br(k(C))$ we have
$$
{}_n\iota^{-1}({}_n\Br(k(C))_{\dot{V}}) = {}_n\Br(k)_V.
$$
\end{prop}

This statement is usually used in conjunction with the following well-known result (cf. \cite[Prop. A.9.1.6]{HinSilv}, \cite{Shim}).
\begin{prop}\label{P:RC334}
Let $k$ be a field equipped with a set $V$ of discrete valuations that satisfies condition {\rm (A)} (see \S\ref{S:Intro}). Given a smooth projective
geometrically irreducible curve $C$ over $k$, there exists a finite subset $V(C) \subset V$ such that $C$ has good reduction at all
$v \in V \setminus V(C)$.
\end{prop}

\medskip

We will next develop a more explicit algebraic version of Proposition \ref{P:RC333}.

\addtocounter{thm}{1}

\medskip

\noindent {\bf \ref{S:Constants}.7. Algebraic statement.}
Let $K = k(x , y)$, where $x$ is transcendental over $k$ and $y$ satisfies a relation $F(x , y) = 0$ with $F(X , Y) \in k[X , Y]$  an absolutely irreducible polynomial in variables $X$ and $Y$ of the form
$$
F(X , Y) = Y^m + f_{m-1}(X)Y^{m-1} + \cdots + f_0(X) \ \ \text{with} \ \ f_i(X) \in k[X].
$$
Furthermore, we let $\delta(X) \in k[X]$ denote the discriminant of $F$ as a polynomial in $Y$, and assume that $\delta \neq 0$. Given a discrete valuation $v$ of $k$, we consider its standard extension $\hat{v}$ to $k(x)$ given on nonzero polynomials by
\begin{equation}\label{E:Val1}
\hat{v}(a_mx^m + \cdots + a_0) = \min_{a_i \neq 0} v(a_i)
\end{equation}
(cf. \cite[Ch. 6, \S 10]{Bour}), and let $\tilde{v}$ denote an arbitrary extension of $\hat{v}$ to $K$.
\begin{prop}\label{P:RC777}
Let ${}_n\iota_k \colon {}_n\Br(k) \to {}_n\Br(K)$ be the canonical map. If  $\hat{v}(f_i(x)) \geqslant 0$ for all $i = 0, \ldots ,
m-1$, and $\hat{v}(\delta(x)) = 0$, then $$r :=
[{}_n\iota_k^{-1}({}_n\Br(K)_{\{ \tilde{v} \}}) : {}_n\Br(k)_{\{ v
\}}]$$  is finite. Moreover, if the reduction $F^{(v)}(X , Y) \in k^{(v)}[X , Y]$ is
absolutely irreducible, then $r = 1$, and in this case $\hat{v}$ extends to $K$ uniquely.
\end{prop}
\begin{proof}
Let $\hat{\mathcal{O}}$ be the valuation ring of $\hat{v}$, and let
$\tilde{\mathcal{O}}$ be the integral closure of $\hat{\mathcal{O}}$
in $K$. It follows from our assumptions that $y \in \tilde{\mathcal{O}}$.
Moreover, the discriminant of the basis $1, y, \ldots , y^{n - 1}$ of $K$
over $k(x)$ is $\delta$, hence a unit in $\hat{\mathcal{O}}$,
implying that

\vskip1mm

\noindent (a) $\tilde{\mathcal{O}} = \hat{\mathcal{O}}[y]$, and

\vskip1mm

\noindent (b) {\it any} extension of $\hat{v}$ to $K$ (in particular, $\tilde{v}$)
is unramified

\vskip2mm

\noindent (cf. \cite[Ch. I, Theorems 7.3 and 7.5]{Jan}). Since the residue field
$k(x)^{(\hat{v})}$ coincides with $k^{(v)}(\bar{x})$, it follows from (a) that the
residue field $K^{(\tilde{v})}$ is $k^{(v)}(\bar{x} ,
\bar{y})$, where $\bar{x} , \bar{y}$ denote the images of $x$ and $y$
(these obviously satisfy the relation $F^{(v)}(\bar{x} , \bar{y}) =
0$). As $\hat{v}(k(x)^{\times}) =
v(k^{\times})$, applying (b) we obtain that  the
ramification index $e = [\tilde{v}(K^{\times}) : v(k^{\times})]$ is $1$. Obviously, $[K^{(\tilde{v})} \cap
(k^{(v)})^{\mathrm{sep}} : k^{(v)}] < \infty$, so from
Corollary \ref{C:RC1}, we conclude that $r < \infty$. If $F^{(v)}$
is absolutely irreducible over $k^{(v)}$, then $$K^{(\tilde{v})} \cap
(k^{(v)})^{\mathrm{sep}} = k^{(v)}$$ (see \cite[\S 3.2, Corollary 2.14]{Liu}), and again  Corollary
\ref{C:RC1} implies that $r = 1$. In fact,  in this
case we have $[K^{(\tilde{v})} : k(x)^{(\hat{v})}] = [K : k(x)]$,
and therefore $\tilde{v}$ is the \emph{unique} extension of
$\hat{v}$.
\end{proof}

Next, we will need to recall the following well-known consequence of ``elimination theory" (cf. \cite{Shim}),
which in fact plays a crucial role in the proof of Proposition \ref{P:RC334}.

\begin{lemma}\label{L:RC3}
Let $k$ be a field equipped with a set $W$ of discrete valuations
that satisfies condition {\rm (A)}. If $F \in k[X , Y]$ is an absolutely irreducible
polynomial, then for almost all $v \in W$,  the reduction $F^{(v)} \in
k^{(v)}[x , y]$ is defined and is absolutely irreducible.
\end{lemma}

Now Proposition \ref{P:RC777} and Lemma \ref{L:RC3} together yield

%

\begin{prop}\label{P:RC1}
Let $K = k(x , y)$ where $x$ is transcendental over $k$ and $F(x , y) = 0$ with $F \in k[X , Y]$
an absolutely irreducible polynomial of the form
$$
F(X , Y) = Y^m + f_{m-1}(X)Y^{m-1} + \cdots + f_0(X), \ \ \ f_i(X) \in k[X].
$$
Let $\delta(X) \in k[X]$ be the discriminant of $F$ as a polynomial in $Y$, and assume that $\delta(X) \neq 0$. Furthermore, let $W$ be a set of discrete valuations of $k$ that
satisfies condition {\rm (A)}. For each $v \in W$, we consider the extension $\hat{v}$ to $k(x)$ given by {\rm (\ref{E:Val1})},
pick one extension\footnotemark \: $\tilde{v}$ of $\hat{v}$ to $K$, and set $\tilde{W} = \{ \tilde{v} \vert v \in W \}$. If
$$
\hat{v}(f_i(x)) \geqslant 0 \ \ \text{for all} \ \ i = 0, \ldots , m-1, \ \ \text{and} \ \ \hat{v}(\delta(x)) = 0,
$$
for all $v \in W$ then
$$
r_W := [{}_n\iota_k^{-1}({}_n\Br(K)_{\tilde{W}}) : {}_n\Br(k)_W] < \infty.
$$
In fact, $r_W = 1$ if the reduction $F^{(v)}(X , Y) \in k^{(v)}$ is absolutely irreducible for all $v \in W$.
\end{prop}
\footnotetext{As follows from Proposition \ref{P:RC777} and Lemma \ref{L:RC3}, for almost all $v \in W$, the extension $\tilde{v}$ is unique.}

\begin{proof}
We have an injective map
\begin{equation}\label{E:RC778}
{}_n\iota_k^{-1}({}_n\Br(K)_{\tilde{W}}) / {}_n\Br(k)_W \to \prod_{v \in W} \iota_k^{-1}({}_n\Br(K)_{\{\tilde{v}\}}) / {}_n\Br(k)_{\{v\}}.
\end{equation}
By Lemma \ref{L:RC3}, the subset $W_0$ of $v \in W$ such that the reduction $F^{(v)}(X , Y)$ fails to be absolutely irreducible, is finite. Invoking Proposition \ref{P:RC777}, we see that the quotient ${}_n\iota_k^{-1}({}_n\Br(K)_{\{\tilde{v}\}}) / {}_n\Br(k)_{\{v\}}$ is finite for all $w \in W$ and is trivial
for $v \in W \setminus W_0$. Thus, the target in (\ref{E:RC778}) is always finite, and in fact is trivial if $F^{(v)}$ is absolutely irreducible for all $v \in W$, hence our assertions.
\end{proof}

\medskip

\addtocounter{thm}{1}

\noindent {\bf \ref{S:Constants}.11. Purely transcendental extensions.} Let $\ell = k(x_1, \ldots , x_r)$ be a purely transcendental extension of a field $k$,
and let $n > 1$ be an integer prime to $\mathrm{char} \: k$. For $i \in \{1, \ldots , r \}$, we set $\ell_i := k(x_1, \ldots , x_{i-1}, x_{i+1}, \ldots , x_r)$ and let $V_i$ denote the set of discrete valuations of $\ell = \ell_i(x_i)$ that correspond to the irreducible polynomials in $\ell_i[x_i]$ (and thus are trivial on $\ell_i$). We define
$$
V = \bigcup_{i = 1}^r V_i.
$$
(Geometrically, one interprets $\ell$ as the field of rational functions on the projective space $\PP_k^r$, and thinks of each $V_i$ as a subset of the set $V_0$ of
discrete valuations of $\ell$ associated with the prime divisors, which, of course, correspond to irreducible homogeneous polynomials in $k[X_0, X_1, \ldots , X_r]$. Then
the union above is taken inside $V_0$; note that $V_i$ and $V_j$ for $i \neq j$ are not disjoint. Obviously, $V$ depends on the choice of a transcendence basis $x_1, \ldots , x_r$.)      
Furthermore, let $T$ be a set of discrete valuations of $k$ such that $(n , \mathrm{char}\: k^{(v)}) = 1$ for all $v \in T$. For each $v$, we let $\check{v}$ denote its natural extension to $\ell$ given on nonzero polynomials by
\begin{equation}\label{E:Val777}
\check{v}\left(\sum_{i_1, \ldots , i_r}
a_{i_1, \ldots , i_r} x_1^{i_1} \cdots x_r^{i_r} \right) = \min_{a_{i_1, \ldots, i_r} \neq 0}
v(a_{i_1, \ldots, i_r}),
\end{equation}
and let $\check{T} = \{ \check{v} \: \vert \: v \in T \}$. Now, set
\begin{equation}\label{E:Val}
V(T) = V \bigcup \check{T}.
\end{equation}
Clearly, $V(T)$ satisfies condition (A) if $T$ does. Moreover, we have the following well-known statement (cf. the proof of \cite[Proposition 3.4]{RR}).
\begin{prop}\label{P:RC2-1}
${}_n\Br(\ell)_{V(T)} = {}_n\Br(k)_T$.
\end{prop}
\begin{proof}
This is an easy consequence of the following two standard facts. Let $F = k(x)$ be the field of rational functions over $k$, and $V^F$ be the set of discrete
valuations of $F$ corresponding to the monic irreducible polynomials in $k[x]$. Furthermore, for a discrete valuation $v$ of $k$, we let $\check{v}$ denote its natural extension to $F$ given by (\ref{E:Val1}).  Then

\medskip

\noindent (a) ({\it assuming that $(n , \mathrm{char} \: k) = 1$}) ${}_n\Br(F)_{V^F} = {}_n\Br(k)$, cf. \cite[\S19.5]{P}

\vskip1mm

\noindent (b) \parbox[t]{15cm}{({\it assuming that either $k^{(v)}$ is perfect or $(n , \mathrm{char} \: k^{(v)}) = 1$}) a central division $k$-algebra $\Delta$ of degree dividing $n$ is unramified at $v$ if and only if the algebra $\Delta \otimes_k F$ is unramified at $\check{v}$. Indeed, this follows from the diagram (\ref{E:RC777}) written for $\ell = F$ and $w = \check{v}$ since $\check{v} \vert v$ is unramified and $F^{(\check{v})} = k^{(v)}(x)$, hence $k^{(v)}$ is algebaically closed in $F^{(\check{v})}$.}

\vskip2mm

\noindent Let $V'$ be the set of discrete valuations of $\ell_r$ constructed in the same manner as $V$ was constructed for $\ell$; clearly, $V'$ is made up of
the restrictions of places from $\bigcup_{i=1}^{r-1} V_i$. It follows from (a) and (b) that ${}_n\Br(\ell)_{V} = {}_n\Br(\ell_r)_{V'}$, so by induction on $r$ we conclude that
$$
{}_n\Br(\ell)_{V} = {}_n\Br(k).
$$
Applying (b) again, we obtain
$$
{}_n\Br(\ell)_{V(T)} = {}_n\Br(k)_T,
$$
as required.
\end{proof}

\medskip

\addtocounter{thm}{1}

\noindent {\bf \ref{S:Constants}.13. Global fields.} We now recall the well-known description of the Brauer group of a global field. So, let $k$ be a global field, $V^k$ be the set of all places of $k$, and let $V^k_{\infty} \supset V^k_{\mathrm{real}}$ be the subsets of archimedean and real places, respectively (of course, $V^k_{\infty} = \emptyset$ if $k$ has positive characteristic). According to the theorem of Albert-Hasse-Brauer-Noether (cf. \cite[18.4]{P}, \cite[Corollary 6.5.4]{GiSz}), there is an exact sequence
$$
0 \to \Br(k) \longrightarrow \bigoplus_{v \in V^k} \Br(k_v) \stackrel{\mathrm{inv}}{\longrightarrow} \Q/\Z \to 0,
$$
where $\mathrm{inv}$ is the sum of the \emph{local invariant maps} $\mathrm{inv}_v \colon \Br(k_v) \to \Q/\Z$. Furthermore, $\mathrm{inv}_v$ is injective for all
$v \in V$, and in fact is an isomorphism for $v$ nonarchimedean, identifies $\Br(k_v)$ with $(1/2)\Z/\Z$ for $v \in V^k_{\mathrm{real}}$ and is trivial for $v \in
V^k_{\infty} \setminus V^k_{\mathrm{real}}$. Moreover, for $v$ nonarchimedean, $\mathrm{inv}_v$ actually coincides with the corresponding residue map, so the unramified
Brauer group $\Br(k_v)_{\{ v \}}$ is trivial. Now, let $S \subset V^k$ be a finite subset containing $V^k_{\infty}$, and set $T = V^k \setminus S$. Then the unramified Brauer
group ${}_n\Br(k)_T$ is identified with the kernel of
$$
\bigoplus_{v \in S} {}_n\Br(k_v) \stackrel{i}{\longrightarrow} \Q/\Z,
$$
where $i$ is the sum of $\mathrm{inv}_v$ for $v \in S$. The above description of $\Br(k_v)$ implies that ${}_n\Br(k_v)$ is $(1/n)\Z/\Z$ for $v$ nonarchimedean, $(n , 2)^{-1}\Z/\Z$
for $v \in V^k_{\mathrm{real}}$, and trivial for $v \in V^k_{\infty} \setminus V^k_{\mathrm{real}}$. On the other hand, $\mathrm{Im}\: i$ is $(1/n)\Z/\Z$
if $S \neq V^k_{\infty}$, is $(n , 2)^{-1}\Z/\Z$ if $S = V^k_{\infty}$ and $V^k_{\mathrm{real}} \neq \emptyset$, and is trivial otherwise. We then obtain the following.
\begin{lemma}\label{L:RC444}
Let $a = \vert V^k_{\mathrm{real}} \vert$ and $b = \vert S \setminus V^k_{\infty} \vert$. Then ${}_n\Br(k)_T$ is a finite group whose order $\beta(n, k, T)$ equals
$(n , 2)^a n^{b-1}$ if $b > 0$, $(n , 2)^{a - 1}$ if $b = 0$ but $a > 0$, and $1$ if $a = b = 0$. (Note that in all cases the order divides $(n , 2)^a n^b$.)
\end{lemma}

\medskip

\addtocounter{thm}{1}

\noindent {\bf \ref{S:Constants}.15. Purely transcendental extensions of global fields.} Combining Proposition \ref{P:RC2-1} with Lemma \ref{L:RC444},
we obtain the following finiteness result.

\begin{prop}\label{P:RC22-1}
Let $\ell = k(x_1, \ldots , x_r)$ be a purely transcendental extension of a global field $k$, let $n > 1$ be an integer prime to $\mathrm{char} \: k$, and
let $T = V^k \setminus S$, where $S \subset V^k$ is a finite subset containing $V^k_{\infty} \cup V(n)$. Then the group ${}_n\Br(\ell)_{V(T)}$ is finite, where $V(T)$ is
as in (\ref{E:Val}).
\end{prop}

\medskip

\noindent {\bf Remark \ref{S:Constants}.17.}
One can show that the group ${}_n\Br(\ell)_{V'}$ remains finite for any subset $V' \subset V(T)$ with finite complement $V(T) \setminus V'$. This result would enable us to somewhat streamline the proof of Theorem 2, but in fact its use can be avoided by a simple algebraic trick which we will describe
in \S\ref{S:FinGenF}.


\medskip

\addtocounter{thm}{2}

\noindent {\bf \ref{S:Constants}.18.  $\iota_k$ vs. ${}_n\iota_k$.} While the application of Proposition \ref{P:Out1} requires information about
$I := \iota^{-1}_k({}_n\Br(k(C))_V)$ for the natural map $\iota_k \colon \Br(k) \to \Br(k(C))$, the results of the current section provide information
about $J = {}_n\iota^{-1}_k({}_n\Br(k(C))_V)$ for the \emph{restriction} ${}_n\iota_k \colon {}_n\Br(k) \to {}_n\Br(k(C))$. So, in this section we will relate the finiteness
of $I$ to that $J$, and more precisely establish a connection between the orders of these groups.
Of course, if $C(k) \neq \emptyset$, then $\iota_k$ is injective, hence $I = J$. In the general
case, the relationship between these groups depends on the following \emph{relative Brauer group}
$$
\Br(k(C)/k) := \mathrm{Ker}\left(\Br(k) \to \Br(k(C))\right).
$$
First, we mention the following elementary group-theoretic lemma.
\begin{lemma}\label{L:RC555}
Let $A$ be an abelian group, $B \subset A$ be a finite subgroup, and $\phi \colon A \to A/B$ be the canonical homomorphism. Then for any integer $n \geqslant 1$,
the index $[{}_n(A/B) : \phi({}_nA)]$ is finite and divides $\vert B \vert$. Thus, if ${}_nA$ is finite then ${}_n(A/B)$ is finite of order dividing $\vert {}_nA \vert
\cdot \vert B \vert$.
\end{lemma}
\begin{proof}
Define $\psi \colon {}_n(A/B) \to B/B^n$ by $aB \mapsto a^nB^n$. Then the sequence
$$
{}_nA \stackrel{\phi}{\longrightarrow} {}_n(A/B) \stackrel{\psi}{\longrightarrow} B/B^n
$$
is exact, and the lemma follows.
\end{proof}

In our situation, this lemma yields the following.
\begin{prop}\label{P:RC000}
Assume that the relative Brauer group $\Br(k(C)/k)$ is finite. If $J$ is finite then $I$ is also finite, with $\vert I \vert$ dividing
$\vert J  \vert \cdot \vert \Br(k(C)/k) \vert^2$. In particular, if the group $\Br(k(C)/k)$ is trivial (which is always the case when
$C(k) \neq \emptyset$), then $I = J$.
\end{prop}
\begin{proof}
By definition, $\Br(k(C)/k)$ coincides with $\mathrm{Ker}\:  \iota_k$. So, applying the lemma to $A = I$ and $B = \Br(k(C)/k)$, we obtain that
the index $[\iota_k(I) : \iota_k(J)]$ divides the order of $\Br(k(C)/k)$, and our assertion follows.
\end{proof}

It was shown by M.~Ciperiani and D.~Krashen \cite[Proposition 4.11]{CiKr} that $\Br(k(C)/k)$ is always finite if $k$ is finitely generated (and not
only for a curve $C$ but for any smooth projective variety). We note that the result in \emph{loc. cit.} is formulated for finitely generated extensions
of $\Q$, however the proof relies only on the finite generation of $\mathrm{Pic}^0(C)(k)$ which holds in any characteristic. Since in the context of  Theorem 2
the field $k$ will be finitely generated, proving the finiteness of $I$ is equivalent to proving the finiteness of $J$. For convenience of reference, we will formulate
the following immediate consequence of Proposition \ref{P:RC333} and Lemma \ref{L:RC444}.
\begin{cor}\label{C:RC999}
Let $C$ be a smooth projective geometrically irreducible curve over a global field $k$ such that $C(k) \neq \emptyset$, and let $T \subset V^k$ be a subset containing $V_{\infty}^k$ with finite complement $V^k \setminus T$. Assume that $C$ has good reduction at every $v \in T$, let $\dot{v}$ denote the canonical extension of $v$ to $k(C)$, and set
$\dot{T} =
\{ \dot{v} \: \vert \: v \in T \}$. Then for the natural map $\iota_k \colon \Br(k) \to \Br(k(C))$ we have
$$
\vert \iota^{-1}_k({}_n\Br(k(C))_{\dot{T}}) \vert = \beta(n, k, T)
$$
(see Lemma \ref{L:RC444} for the definition of $\beta(n, k, T)$).
\end{cor}

\medskip

\section{A result on finitely generated fields}\label{S:FinGenF}

In this section, we will describe a specific presentation of a given finitely generated field as a function field of a geometrically
irreducible smooth projective curve $C$. We will then use  properties of this presentation to verify item (I) of Proposition \ref{P:Out1}
for an appropriate set of valuations.
\begin{prop}\label{P:FGF1}
Let $K$ be a finitely generated field which is not a global field. There exists a global field $P \subset K$ and elements $t_1, \ldots , t_r, x, y \in K$ such that

\medskip

{\rm (1)} $K = P(t_1, \ldots , t_r, x, y)$;

\vskip1mm

{\rm (2)} $t_1, \ldots , t_r$ and $x$ are algebraically independent over $P$;

\vskip1mm

{\rm (3)} $k := P(t_1, \ldots , t_r)$ is algebraically closed in $K$;

\vskip1mm

{\rm (4)} \parbox[t]{15cm}{there is a polynomial $f(T_1, \ldots , T_r, X , Y) \in P[T_1, \ldots , T_r, X, Y]$ of the form
$$
f = Y^d + p_{d-1}(T_1, \ldots , T_r, X)Y^{d-1} + \cdots + p_0(T_1, \ldots , T_r, X)
$$
that satisfies the following properties: \vskip2mm {\rm (a)} $f(t_1, \ldots , t_r, x, y) = 0$; \vskip1mm  {\rm (b)} $f(t_1, \ldots , t_r, x, Y) \in k(x)[Y]$ is
irreducible; \vskip1mm  {\rm (c)} the discriminant $\delta_Y(f)$ of $f$ as a polynomial in $Y$ has the form
$$
\delta_Y(f) = q_m(T_1, \ldots , T_r)X^m + q_{m-1}(T_1, \ldots , T_r)X^{m-1} + \cdots + q_0(T_1, \ldots , T_r),
$$
with $q_i \in P[T_1, \ldots , T_r]$ and $q_m \in P^{\times}$.}
\end{prop}

\medskip

\addtocounter{thm}{1}

\noindent {\bf \ref{S:FinGenF}.2. Verification of item (I) of Proposition \ref{P:Out1}.} We may assume that the given  finitely generated field $K$ is not global as
otherwise Theorem 2 follows from Lemma \ref{L:RC444}. We fix a presentation of $K$ of the form $K = P(t_1, \ldots , t_r, x, y)$ satisfying the properties
listed in Proposition \ref{P:FGF1}. Set $k = P(t_1, \ldots , t_r)$ and let $C$ be a connected $k$-defined smooth projective curve such that $k(C) = k(x , y) = K$.
Note that due to condition (3) in the proposition, $C$ is geometrically irreducible (see \cite[Ch.3, Corollary 2.14]{Liu}).
As in 3.11, we write $k_i = P(t_1, \ldots , t_{i-1}, t_{i+1}, \ldots , t_r)$,
let $V_i$ be the set of discrete valuations of $k$ corresponding to the irreducible monic polynomials in $k_i[t_i]$, and set
$$
V = \bigcup_{i = 1}^r V_i.
$$
Next, for each discrete valuation $v$ of $P$, we let $\check{v}$ be its natural extension to $k$ given by a formula similar to (\ref{E:Val777}), and further denote by $\hat{v}$
the extension of $\check{v}$ to $k(x)$ given by (\ref{E:Val1}). Now, pick
a set $T \subset V^P$ of discrete valuations with finite complement $V^P \setminus T$ such that for each $v \in T$ we have $(n , \mathrm{char}\: P^{(v)}) = 1$ and
\begin{equation}\label{E:FinGenF1}
\hat{v}(p_i(t_1, \ldots , t_r, x)) \geqslant 0 \ \ \text{for all} \ \ i = 1, \ldots , d-1, \ \ \text{and} \ \ v(q_m) = 0.
\end{equation}
Since the discriminant is given by an integral polynomial in the coefficients of the original polynomial, we have $\hat{v}(\delta_Y(f)) \geqslant 0$, so the condition
$v(q_m) = 0$ forces
\begin{equation}\label{E:FinGenF100}
\hat{v}(\delta_Y(f)) = 0.
\end{equation}
Set $\check{T} = \{ \check{v} \vert v \in T \}$ and $V(T) = V \bigcup \check{T}$. We will still use the notation $\hat{v}$ to denote the extension of every $v \in V(T)$ to $k(x)$
given by (\ref{E:Val1}). Let $\tilde{v}$ denote an arbitrary extension of $\hat{v}$ to $k(x , y) = K$ and set $\widetilde{V(T)} = \{ \tilde{v} \vert v \in V(T) \}$. Then (\ref{E:FinGenF1}), (\ref{E:FinGenF100}) and the fact that $q_m \in P^{\times}$ imply that
$$
\hat{v}(p_i(t_1, \ldots , t_r, x)) \geqslant 0 \ \ \text{for all} \ \ i = 1, \ldots , d-1, \ \ \text{and}  \ \ \hat{v}(\delta_Y(f)) = 0.
$$
for all $v \in V(T)$; thus, the assumptions of Proposition \ref{P:RC1} hold for $W = V(T)$. We then obtain the following.
\begin{thm}\label{T:FinGenF1}
In the above notations, for the canonical map $\iota_k \colon \Br(k) \to \Br(K)$, the group $\iota_k^{-1}({}_n\Br(K)_{\widetilde{V(T)}})$ is finite.
\end{thm}
\begin{proof}
By Proposition \ref{P:RC000} and subsequent remarks, it is enough to prove that the group $${}_n\iota_k^{-1}({}_n\Br(K)_{\widetilde{V(T)}})$$ is finite. Furthermore,
by Proposition \ref{P:RC1}, the index $[{}_n\iota^{-1}_k({}_n\Br(K)_{\widetilde{V(T)}}) : {}_n\Br(k)_{V(T)}]$ is finite. Finally, by Proposition \ref{P:RC22-1}, the group
${}_n\Br(k)_{V(T)}$ is finite, completing the proof.
\end{proof}

\medskip

\addtocounter{thm}{1}

\noindent {\bf \ref{S:FinGenF}.4. Proof of Proposition \ref{P:FGF1}.} We will derive this proposition from the following more general result. Following \cite[Ch. VIII]{La}, we will call a field
extension $K/P$ \emph{regular} if it is separable and $P$ is algebraically closed in $K$; equivalently, $\overline{P}$ and $K$ are linearly disjoint over $P$, for a choice
of algebraic closure $\overline{P}$ of $P$ (see \cite[Ch. VIII, \S4]{La}, particularly pp. 366-367).
\begin{prop}\label{P:FGF2}
Let $K$ be a finitely generated regular field extension of an infinite field $P$ of transcendence degree $> 0$. Then there exist elements $t_1, \ldots , t_r, x, y \in K$
such that

\medskip

{\rm (1)} $K = P(t_1, \ldots , t_r, x, y)$;

\vskip1mm

{\rm (2)} $t_1, \ldots , t_r$ and $x$ are algebraically independent over $P$;

\vskip1mm

{\rm (3)} $k := P(t_1, \ldots , t_r)$ is algebraically closed in $K$;

\vskip1mm

{\rm (4)} \parbox[t]{15cm}{there is a polynomial $f(T_1, \ldots , T_r, X , Y) \in P[T_1, \ldots , T_r, X, Y]$ of the form
$$
f = Y^d + p_{d-1}(T_1, \ldots , T_r, X)Y^{d-1} + \cdots + p_0(T_1, \ldots , T_r, X)
$$
that satisfies the following properties: \vskip2mm {\rm (a)} $f(t_1, \ldots , t_r, x, y) = 0$; \vskip1mm {\rm (b)} $f(t_1, \ldots , t_r, x, Y) \in k(x)[Y]$ is
irreducible; \vskip1mm {\rm (c)} the discriminant $\delta_Y(f)$ of $f$ as a polynomial in $Y$ has the form
$$
\delta_Y(f) = q_m(T_1, \ldots , T_r)X^m + q_{m-1}(T_1, \ldots , T_r)X^{m-1} + \cdots + q_0(T_1, \ldots , T_r),
$$
with $q_i \in P[T_1, \ldots , T_r]$ and $q_m \in P^{\times}$.}
\end{prop}

To derive Proposition \ref{P:FGF1} from this proposition, we let $F$ denote the prime subfield of $K$. Since $F$ is perfect, the extension $K/F$ has a separating
transcendence basis $s_1, \ldots , s_a$ (cf. \cite[Ch. VIII, Proposition 4.1]{La}). Set $E = F$ if the characteristic is zero, and $E = F(s_1)$ if the characteristic is $p > 0$.
Then $K/E$ is separable (cf. \cite[Ch. VIII, Cor. 5.6]{La}). Let $P$ be the algebraic closure of $E$ in $K$. Note that $P/E$ is a finite extension, and therefore $P$ is a global field. Furthermore, the extension $K/P$ is still separable, hence regular, and of transcendence degree $> 0$ as $K$ is not global by assumption. Now, Proposition \ref{P:FGF1}
immediately follows from Proposition \ref{P:FGF2}. $\Box$

\medskip

\addtocounter{thm}{1}

\noindent {\bf \ref{S:FinGenF}.6. Proof of Proposition \ref{P:FGF2}.} We will need the following case of a theorem of Zariski-Matsusaka (cf. \cite[Ch. VIII, Theorem 5.7]{La}).
\begin{prop}\label{P:ZM}
Let $K/k$ be a regular finitely generated field extension, and let $y , z \in K$ be algebraically independent over $k$ such that at least one of them does not
lie in $K^pk$ if the characteristic is $p > 0$. Then for all but finitely many $c \in k$, the field $K$ is regular over $k(y + cz)$.
\end{prop}

\begin{cor}\label{C:FGF1}
Let $x, z_1, \ldots , z_r$ be a separating transcendence basis for a regular field extension $K$ of an infinite field $P$. Then for any positive integers $m_1, \ldots ,
m_r$ and any infinite subset $\mathcal{P} \subset P$ there exist $c_1, \ldots , c_r \in \mathcal{P}$ such that for the elements
$$
t_1 = z_1 + c_1x^{m_1}, \ldots ,  t_r = z_r + c_rx^{m_r},
$$
the field $K$ is a regular extension of $P(t_1, \ldots , t_r)$.
\end{cor}
\begin{proof}
We will induct on $r$, observing that there is nothing to prove if $r = 0$. So, let $r > 0$. Clearly, $z_1 \notin K^p P$ if the characteristic is $p > 0$, so
by Proposition \ref{P:ZM} there exists $c_1 \in \mathcal{P}$ such that $K$ is regular over $P_1 := P(t_1)$ where $t_1 = z_1 + c_1x^{m_1}$. Then $x, z_2, \ldots ,z_r$ is
a separating transcendence basis of $K$ over $P_1$, and by  induction there exist $c_2, \ldots , c_r \in \mathcal{P} \subset P_1$ such that for $t_2 = z_2 +
c_2x^{m_2}, \ldots , t_r = z_r + c_rx^{m_r}$, the field $K$ is regular over $P_1(t_2, \ldots , t_r) = P(t_1, \ldots , t_r)$.
\end{proof}

Turning now to the proof of Proposition \ref{P:FGF2} (which imitates the proof of Noether's Normalization Lemma), we first observe that since  $K/P$ is
a finitely generated regular field extension of transcendence degree $> 0$, one can
find a separating transcendence basis $x, z_1, \ldots , z_r$. Then there exists $y \in K$ such that $K = P(x, z_1, \ldots , z_r, y)$. Multiplying $y$ by a suitable
element of $P[x, z_1, \ldots , z_r]$, we may assume that there exists an irreducible polynomial $\bar{f} \in P[X, Z_1, \ldots , Z_r, Y]$ of the form
$$
\bar{f} = Y^d + \bar{p}_{d-1}(X, Z_1, \ldots , Z_r)Y^{d-1} + \cdots + \bar{p}_0(X, Z_1, \ldots , Z_r)
$$
such that $\bar{f}(x, z_1, \ldots , z_r, y) = 0$. Then $\bar{f}(x, z_1, \ldots , z_r, Y) \in P(x, z_1, \ldots , z_r)[Y]$ is the minimal polynomial for $y$ over
$P(x, z_1, \ldots , z_r)$, hence separable, and therefore the discriminant $\delta_Y(\bar{f}) \in P[X, Z_1, \ldots , Z_r]$ is nonzero. If $r = 0$ then $K = P(x , y)$, and there is nothing to prove, so we assume that $r > 0$. Fix positive integers $m_1, \ldots , m_r$. By Corollary \ref{C:FGF1}, we can pick \emph{nonzero} $c_1, \ldots , c_r \in P$
such that for
$$
t_1 = z_1 - c_1x^{m_1}, \ldots , t_r = z_r - c_r x^{m_r},
$$
the field $K$ is a regular extension of $P(t_1, \ldots , t_r)$. Then properties (1)-(3) in the statement of the proposition are clearly satisfied, and it remains
to pick the integers $m_1, \ldots , m_r$ so as to satisfy property (4) as well (the choice of the $c_i$'s will play no role as long as they are nonzero). We introduce
new variables $T_i = Z_i - c_i X^{m_i}$, $i = 1, \ldots , r$. Since $P[X, Z_1, \ldots , Z_r] = P[X, T_1, \ldots , T_r]$, we can express $\bar{p}_0, \ldots , \bar{p}_{d-1}$
(resp., $\bar{f}$) in terms of $X, T_1, \ldots , T_r$ (resp., $X, T_1, \ldots , T_r, Y$), and we will denote the resulting polynomials by $p_0, \ldots , p_{d-1}$ (resp.,
$f$). Clearly, we have
$$
f = Y^d + p_{d-1}(X, T_1, \ldots , T_r)Y^{d-1} + \cdots + p_0(X, T_1, \ldots , T_r) \ \ \text{and} \ \ \delta_Y(f) = \delta_Y(f).
$$
Write
$$
\delta_Y(\bar{f}) = \sum_{(i_0, i_1, \ldots , i_r) \in I} a_{i_0i_1 \cdots i_r} X^{i_0} Z_1^{i_1} \cdots Z_r^{i_r} \ \ \text{with} \ \ a_{i_0i_1 \ldots i_r} \neq 0,
$$
where the sum is taken over a finite subset $I \subset (\mathbb{N} \cup \{ 0 \})^{r+1}$. Pick an integer $\ell$ so that $\ell \geqslant i_j$ for every $(i_0, i_1, \ldots , i_r)
\in I$ and all $j = 0, \ldots , r$, and then set  $m_j = (\ell + 1)^j$ for $j = 1, \ldots , r$. We have
$$
\delta_Y(f) = \sum_{(i_0, i_1, \ldots , i_r) \in I} a_{i_0i_1 \cdots i_r} X^{i_0} (T_1 + c_1X^{m_1})^{i_1} \cdots (T_r + c_rX^{m_r})^{i_r}.
$$
Note that
$$
X^{i_0} (T_1 + c_1X^{m_1})^{i_1} \cdots (T_r + c_rX^{m_r})^{i_r} = \left( \prod_{j = 1}^r c_j \right) \cdot X^{d(i_0, i_1, \ldots , i_r)} + \text{(terms of lower}\,X\text{-degree)},
$$
where
$$
d(i_0, i_1, \ldots , i_r) := i_0 + i_1(\ell+1) + \cdots + i_r(\ell+1)^r.
$$
It easily follows from our choice of $\ell$ that the numbers $d(i_0, i_1, \ldots , i_r)$ for $(i_0, i_1, \ldots , i_r) \in I$ are all distinct, and in particular
there exists only one $(i^0_0, i^0_1, \ldots i^0_r) \in I$ for which $d(i^0_0, i^0_1, \ldots i^0_r)$ equals
$$
m := \max_{(i_0, i_1, \ldots , i_r) \in I} d(i_0, i_1, \ldots , i_r).
$$
Then
$$
\delta_Y(f) = a_{i_0i_1 \cdots i_r} \cdot \left( \prod_{j=1}^r c_j \right) \cdot X^m + q_{m-1}(T_1, \ldots , T_r) X^{m-1} + \cdots ,
$$
as required. $\Box$

\vskip5mm

\section{Finiteness of unramified cohomology in degree one}\label{S:UnramClass}

As we indicated in \S \ref{S:Intro}, item (II) of Proposition \ref{P:Out1}  will be verified by employing a suitable modification of the strategy used in the proof
of the weak Mordell-Weil theorem. This strategy is based on the consideration of unramified cohomology classes, and in this section we recall the relevant definitions
and discuss the basic set-up.

Let $\ell/k$ be a Galois extension with Galois group $\Gamma = \Ga(\ell/k)$. For a $\Gamma$-module $M$, the cohomology group $H^i(\Gamma , M)$ will, as usual,  be denoted by $H^i(\ell/k , M)$. Furthermore, we will write $M(k)$ for $M^{\Gamma}$ and $H^i(k , M)$ for $H^i(k^{\mathrm{sep}}/k , M)$, where $k^{\mathrm{sep}}$ is a (fixed) separable
closure of $k$. Given a discrete valuation $v$ of $k$, we let $k_v^{\mathrm{ur}}$ denote the maximal unramified extension of the completion $k_v$. Let $M$ be a finite $\Ga(k^{\mathrm{sep}}_v/k_v)$-module. We say that a cohomology class $x \in H^i(k_v , M)$ $(i \geqslant 1)$ is  \emph{unramified} if it lies in the image of the inflation map $
H^i(k^{\mathrm{ur}}_v/k_v , M(k^{\mathrm{ur}}_v)) \to H^i(k_v , M)$.
In the current paper, we will use this notion only for $i = 1$ where it is equivalent, by the inflation-restriction exact sequence,  to the fact that $x$ lies in the kernel of the restriction map $H^1(k_v , M) \to
H^1(k^{\mathrm{ur}}_v , M)$. Furthermore, for a $\Ga(k^{\mathrm{sep}}/k)$-module $M$, a class $x \in H^i(k , M)$ is unramified at $v$ if its image under the restriction
map $H^i(k , M) \to
H^i(k_v , M)$ is unramified. Finally, given a set $T$ of discrete valuations of $k$, we say that $x \in H^i(k , T)$ is unramified at $T$ if it is unramified at every $v \in T$. The subgroup of all classes in $H^i(k , M)$ that are unramified at $T$ will be denoted by $H^i(k , M)_T$.

\vskip2mm

In order to formulate the conditions that guarantee the finiteness of $H^1(k , M)_T$ in certain situations, we need to introduce the following definitions.
Let $T$ be a set of discrete valuations of $k$ satisfying condition (A) (see \S \ref{S:Intro}). We let $\mathrm{Div}(T)$
denote the free abelian group on the set $T$, the elements of which
will be called ``{\it divisors}."  Since $T$ satisfies (A), to any
$a \in k^{\times}$ we can associate the ``{\it principal divisor}"
$$
(a) = \sum_{v \in T} v(a) \cdot v,
$$
and we let $\mathrm{P}(T)$ denote the subgroup of $\mathrm{Div}(T)$
formed by all principal divisors. The quotient
$\mathrm{Div}(T)/\mathrm{P}(T)$ will be called the {\it Picard
group} of $T$ and denoted by $\mathrm{Pic}(T)$ (or $\mathrm{Pic}(k , T)$ if the field $k$ needs to be specified).
Furthermore, we define the {\it group of $T$-units} $\mathrm{U}(T)$ (or $\mathrm{U}(k , T)$) by
$$
\mathrm{U}(T) = \{ a \in k^{\times} \: \vert \: v(a) = 0 \ \
\text{for all} \ \ v \in T \} = \bigcap_{v \in T} \cO_{k ,
v}^{\times}.
$$
Now, for an integer $n > 1$, we introduce the following finiteness condition

\medskip

\noindent $(\mathbf{F}(k , T))_n$ \  {\it the groups $\mathrm{U}(k , T)/\mathrm{U}(k , T)^n$ and ${}_n\mathrm{Pic}(k , T)$ are finite.}

\medskip

\noindent We will also consider the condition

\medskip

\noindent  $(\mathbf{F}(k , T))$ \  {\it the groups $\mathrm{U}(k , T)$ and $\mathrm{Pic}(k , T)$ are finitely generated.}

\medskip

\noindent Obviously, $(\mathbf{F}(k , T))$ implies $(\mathbf{F}(k , T))_n$ for all $n > 1$. The condition $(\mathbf{F}(k , T))_n$  plays the key role
in the following.
\begin{prop}\label{P:UR777}
Let $n > 1$ be an integer prime to $\mathrm{char}\: k$. Assume that  condition $(\mathbf{F}(k , T))_n$ holds.

\medskip

\noindent {\rm (1)} \parbox[t]{15cm}{If $k$ contains a primitive $n$th root of unity, then the group $H^1(k , \mu_n)_T$ is finite of order dividing
$\vert \mathrm{U}(T) / \mathrm{U}(T)^n \vert \cdot \vert {}_n\mathrm{Pic}(T) \vert$.}

\medskip

\noindent {\rm (2)} \parbox[t]{15cm}{If $n$ is prime to $\mathrm{char}\: k^{(v)}$ for all $v \in T$, then $H^1(k , \mu_n)_T$ is finite
of order equal to $\vert \mathrm{U}(T) / \mathrm{U}(T)^n \vert \cdot \vert {}_n\mathrm{Pic}(T) \vert$.}
\end{prop}
\begin{proof}
It is well-known that there is an isomorphism
\begin{equation}\label{E:UR777}
H^1(k , \mu_n) \simeq k^{\times} / {k^{\times}}^n
\end{equation}
under which a coset $a {k^{\times}}^n$ corresponds to the  class of the cocycle $\chi_a$ given by
$$
\chi_a(\sigma) = \frac{\sigma\left( \sqrt[n]{a} \right)}{\sqrt[n]{a}} \ \ \text{for} \ \ \sigma \in \Ga(k^{\mathrm{sep}}/k)
$$
(the cohomology class obviously does not depend on the choice of $\sqrt[n]{a}$). In both cases (1) and (2), for any $v \in T$ the field
$k^{\mathrm{ur}}_v$ contains a primitive $n$th root of unity. It follows that the class of $\chi_a$ is unramified at $v$,
i.e. lies in the kernel of the restriction map $H^1(k , \mu_n)
\to H^1(k^{\mathrm{ur}}_v , \mu_n)$, if and only if $\chi_a$ vanishes on $\Ga(k^{\mathrm{sep}}_v/k^{\mathrm{ur}}_v)$, or equivalently, $\sqrt[n]{a} \in k^{\mathrm{ur}}_v$.
A necessary condition for this is $v(a) \equiv 0(\mathrm{mod}\: n)$. In fact, this condition is also sufficient if $(n , \mathrm{char} \: k^{(v)}) = 1$. This leads to the following
description of $H^1(k , \mu_n)_T$. Let
$$
\mathrm{E}(T , n) = \{ a \in k^{\times} \: \vert \: v(a) \equiv 0(\mathrm{mod}\: n) \ \ \text{for all} \ \ v \in T \},
$$
and let $\mathrm{D}(T , n) = \nu_n(\mathrm{E}(T , n))$, where $\nu_n \colon k^{\times} \to k^{\times} / {k^{\times}}^n$ is the canonical homomorphism. Then the isomorphism (\ref{E:UR777}) takes $H^1(k , \mu_n)_T$ into $D(T , n)$ in case (1), and gives an isomorphism between these two groups in case (2). Thus, to complete the proof of the proposition, it suffices to show  that $D(T , n)$ is finite of order $\vert \mathrm{U}(T) / \mathrm{U}(T)^n \vert \cdot \vert {}_n\mathrm{Pic}(T) \vert$. This fact follows from the existence
of the following exact sequence
\begin{equation}\label{E:UR788}
0 \to \mathrm{U}(T)/\mathrm{U}(T)^n
\stackrel{\alpha}{\longrightarrow} \mathrm{D}(T , n) \stackrel{\beta}{\longrightarrow}
{}_n\mathrm{Pic}(T) \to 0
\end{equation}
(cf. \cite[Ch. 6, Theorem 1.4]{LaDG}).
For the sake of completeness, we briefly recall the construction of
(\ref{E:UR788}). The map $\alpha$ is induced by the identity embedding
$\mathrm{U}(T) \hookrightarrow k^{\times}$; it is clearly injective and its
image is contained in $\mathrm{D}(T , n)$. To define $\beta$, we
observe that since the abelian group $\mathrm{Div}(T)$ is
torsion-free, for any $a \in \mathrm{E}(T , n)$ there exists a
unique $d(a) \in \mathrm{Div}(T)$ such that
\begin{equation}\label{E:UR799}
(a) = n \cdot d(a).
\end{equation}
Then $\beta$  sends $\nu_n(a)$ to the class $p(a) \in
\mathrm{Pic}(T)$ of $d(a)$. It follows from (\ref{E:UR799}) that
$p(a)$ depends only on $\nu_n(a)$ and belongs to
${}_n\mathrm{Pic}(T)$. Conversely, given any $p \in
{}_n\mathrm{Pic}(T)$, for a divisor $d \in \mathrm{Div}(T)$
representing $p$ we have $n \cdot d = (a)$ for some $a \in
k^{\times}$. Then $a \in \mathrm{E}(T , n)$ and $\beta(\nu_n(a)) =
p$, i.e. $\beta$ is surjective. Finally, suppose $b = \nu_n(a) \in
\mathrm{Ker}\: \beta$. Then
$$
(a) = n \cdot (c) \ \ \text{for some} \ \ c \in k^{\times}.
$$
It follows that $(a \cdot c^{-n}) = 0$, i.e. $a \cdot c^{-n} \in
\mathrm{U}(T)$, and
$$
b = \nu_n(a) = \nu_n(a \cdot c^{-n}),
$$
proving the exactness of (\ref{E:UR788}) in the middle term.
\end{proof}

\medskip

\addtocounter{thm}{1}

\noindent {\bf \ref{S:UnramClass}.2. Arbitrary finite module of coefficients.} We will now extend the above finiteness result to the case of an arbitrary finite module $M$
and estimate the order of $H^1(k , M)_T$.  Given a field extension $\ell/k$  and a set
$T_{\ell}$  of discrete valuations of $\ell$, we will say that $T_{\ell}$
{\it lies above} $T$, and write $T_{\ell} \mid T$, if the restriction to $k$ of
every $w \in T_{\ell}$ lies in $T$. Let now $M$ be a finite $\Ga(k^{\mathrm{sep}}/k)$-module, $\ell/k$  a finite
Galois extension, and $T_{\ell} \mid T$.
\begin{lemma}\label{L:UR777}
Assume that $H^1(\ell , M)_{T_{\ell}}$ is finite. Then $H^1(k , M)_T$ is also finite of order dividing $\vert H^1(\ell/k , M(\ell)) \vert \cdot \vert
H^1(\ell , M)_{T_{\ell}} \vert$.
\end{lemma}

The proof immediately follows from the inflation-restriction exact sequence
$$
0 \to H^1(\ell/k , M(\ell)) \stackrel{\mathbf{i}}{\longrightarrow} H^1(k , M) \stackrel{\mathbf{r}}{\longrightarrow} H^1(\ell , M)
$$
as $\mathbf{r}$ takes $H^1(k , M)_{T}$ to $H^1(\ell , M)_{T_{\ell}}$.

\medskip

Combining Proposition \ref{P:UR777} with Lemma \ref{L:UR777}, we obtain the following.
\begin{prop}\label{P:UR778}
Let $M$ be a finite $\Ga(k^{\mathrm{sep}}/k)$-module that has $d$ generators and exponent dividing $n$ with $(n , \mathrm{char}\: k) = 1$ as an abelian group.
Assume that there exists a finite Galois extension $\ell$ equipped with a set $T_{\ell}$ of discrete valuations lying above $T$ such that

\vskip2mm

\noindent {\rm (a)} $\ell$ contains a primitive $n$th root of unity and $M = M(\ell)$,

\vskip1mm

\noindent {\rm (b)} condition $(\mathbf{F}(\ell , T_{\ell}))_n$ holds.

\vskip2mm

\noindent Then $H^1(k , M)_{T}$ is finite of order dividing
\begin{equation}\label{E:UR888}
\vert H^1(\ell/k , M(\ell)) \vert \cdot \left( \vert \mathrm{U}(\ell , T_{\ell})/\mathrm{U}(\ell , T_{\ell})^n \vert \cdot \vert {}_n\mathrm{Pic}(\ell , T_{\ell})
\vert \right)^d.
\end{equation}
\end{prop}
Indeed, by our assumption
$$
M \simeq \Z/n_1\Z \times \cdots \times \Z/n_d\Z \ \ \text{for some} \ \ n_i \vert n
$$
as abelian groups. Then it follows from (a) that
$$
M \simeq \mu_{n_1} \times \cdots \times \mu_{n_d}
$$
as $\Ga(k^{\mathrm{sep}}/\ell)$-modules. Since for $m \vert n$, the orders of $\mathrm{U}(\ell , T_{\ell})/\mathrm{U}(\ell , T_{\ell})^m$ and ${}_m\mathrm{Pic}(\ell , T_{\ell})$
are finite and divide the orders of $\mathrm{U}(\ell , T_{\ell})/\mathrm{U}(\ell , T_{\ell})^n$ and ${}_n\mathrm{Pic}(\ell , T_{\ell})$, respectively, our assertion follows from the previous statements.

\medskip

\addtocounter{thm}{1}

\noindent {\bf \ref{S:UnramClass}.5. Explicit estimations.} Let $k$ be a global field, $S \subset V^k$   a finite nonempty
subset containing $V_{\infty}^k$, and $T := V^k \setminus S$. Then
$\mathrm{U}(T)$  coincides with the group of units $\mathcal{O}_k(S)^{\times}$ of the ring $\mathcal{O}_k(S)$ of $S$-integers in $k$, hence is isomorphic
to $\mu(k) \times \Z^{\vert S \vert - 1}$, where $\mu(k)$ is the group of roots of unity in $k$ (Dirichlet's theorem, cf. \cite[Ch. 2, \S18]{ANT} or
\cite[Prop. 5.17]{RamVal}). Furthermore,  $\mathrm{Pic}(T)$ coincides with the class group of  $\mathcal{O}_k(S)$, hence finite
(cf. \cite[Ch. 2, \S17]{ANT} or \cite[Theorem 5.18]{RamVal}). In particular, condition $(\mathbf{F}(k , T))$ holds; in fact,
$$
\vert \mathrm{U}(T)/\mathrm{U}(T)^n \vert = (n , \vert \mu(k) \vert) \cdot n^{\vert S \vert - 1},
$$
while ${}_n\mathrm{Pic}(T)$ is the $n$-torsion subgroup of the class group of $\mathcal{O}_k(S)$, so has order dividing
$h_k(S , n)$, the largest divisor of the class number $h_k(S)$ of $\mathcal{O}_k(S)$ that involves only the prime divisors of $n$.

On the other hand, the order of $H^1(\ell/k , M(\ell))$ divides the order of the group of 1-cocycles $Z^1(\ell/k , M(\ell))$ which in turn can be embedded into the direct product
$M(\ell)^s$ where $s$ is the number of generators of $\Ga(\ell/k)$. This crude estimation can be improved and made explicit in various concrete situations that arise in applications (cf. \S\ref{S:Examples}).

%

\vskip5mm

\section{Back to the Brauer group of a curve: a local computation}\label{S:Loc}

Let $C$ be a smooth projective geometrically irreducible curve over a field $k$, and $n > 1$ be an integer prime to $\mathrm{char}\: k$.
We will follow up on the description of ${}_n\Br(k(C))_{\mathrm{ur}}$ given in \S \ref{S:BrExactSeq} assuming that $C(k) \neq \emptyset$. According
to Remark 2.2, this assumption implies that the natural embedding $\bar{k}^{\times} \to \bar{k}(C)^{\times}$, where $\bar{k} = k^{\mathrm{sep}}$, has a
$k$-defined section $\varepsilon \colon \bar{k}(C) \to \bar{k}^{\times}$, and then (\ref{E:Out2a}) combined with Hilbert's Theorem 90 implies
that for the group $\mathrm{P}(\bar{C})$ of principal divisors on $\bar{C} = C \times_k
\bar{k}$ we have $H^1(k , \mathrm{P}(\bar{C})) = 0$. The vanishing of this group tells us that the map $\mathrm{Div}(\bar{C})(k) \to \mathrm{Pic}(\bar{C})(k)$ is surjective. (Recall that for a $\Gamma$-module $M$, where $\Gamma = \Ga(\bar{k}/k)$, we write $M(k)$ to denote $H^0(\Gamma , M) = M^{\Gamma}$.) Again by Hilbert's Theorem 90, every element in $\mathrm{P}(\bar{C})(k)$ is the divisor of a function in $k(C)^{\times}$. We will now use these facts to construct a pairing
$$
\mu \colon H^1(k , \Z/n\Z) \times H^0(k , {}_n\mathrm{Pic}^0(\bar{C})) \longrightarrow {}_n\Br(k(C))_{\mathrm{ur}}.
$$
Let $\chi \in H^1(k , \Z/n\Z) = \mathrm{Hom}(\Gamma , \Z/n\Z)$ be a character of order $m \vert n$, and let $a  \in {}_n\mathrm{Pic}^0(\bar{C})(k)$. Pick a divisor $\hat{a} \in \mathrm{Div}(\bar{C})(k)$ representing  $a$. Then $n \hat{a} \in
\mathrm{P}(\bar{C})(k)$, and we set $f_a = \epsilon(n\hat{a})$, where $\epsilon \colon \mathrm{P}(\bar{C}) \to \bar{k}(C)^{\times}$ is the embedding given by identifying
$\mathrm{P}(\bar{C})$ with $\mathrm{Ker}\: \varepsilon$. We then consider  the corresponding cyclic algebra $(\chi , f_a)$ of degree dividing $n$ (cf. \cite[\S2.5]{GiSz} for precise definitions), and define $\mu$ by sending the pair $(\chi , a)$ to the class of $[\chi , f_a] \in \Br(k(C))$ of $(\chi , f_a)$.
\begin{prop}\label{P:Loc-Com1}
{\rm (1)} $[\chi , f_a] \in {}_n\Br(k(C))_{\mathrm{ur}}$;

\vskip1mm

{\rm (2)} the diagram
$$
\xymatrix{& H^1(k , \Z/n\Z) \otimes H^0(k , {}_n\mathrm{Pic}^0(\bar{C})) \ar[ld]_{\mu} \ar[rd]^{\nu} & \\ {}_n\Br(k(C))_{\mathrm{ur}} \ar[rr]^(.50){\omega_k} & & {}_n H^1(k , \mathrm{Pic}^0(\bar{C}))}
$$
where $\nu$ is the $\cup$-product $H^1(k , \Z/n\Z) \times H^0(k , {}_n\mathrm{Pic}^0(\bar{C})) \to H^1(k , {}_n\mathrm{Pic}^0(\bar{C}))$ followed by the natural map
$\theta_k \colon H^1(k , {}_n\mathrm{Pic}^0(\bar{C})) \to {}_nH^1(k , \mathrm{Pic}^0(\bar{C}))$, commutes.
\end{prop}
\begin{proof}
(1): Let $x = [\chi , f_a]$. We need to show that $\rho(x) = 0$, where $\rho \colon \Br'(k(C)) \to H^2(k , \mathrm{Div}(\bar{C}))$ is defined by (\ref{E:def-rho-a}). Let $m \vert n$ be the order of $\chi$ as a character of $\Gamma = \Ga(\bar{k}/k)$, and let $\Delta = \mathrm{Ker}\: \chi$ so that $\Gamma/\Delta$ is a cyclic group of order $m$. Clearly, $\mathrm{Im}\: \chi = (n/m)\Z/n\Z$, and we pick $\sigma \in \Gamma$ such that $\chi(\sigma) = (n/m) (\mathrm{mod}\: n)$. Then the coset $\sigma \Delta$ generates $\Gamma / \Delta$,
and $x$ corresponds to the cohomology class in $H^2(k , \bar{k}(C)^{\times})$ given by the following cocycle on $\Gamma / \Delta$:
\begin{equation}\label{E:LocCom123}
\xi(\sigma^i \Delta , \sigma^j \Delta) = \left\{ \begin{array}{ccl} 1 & , & i + j < m \\ f_a & , & i + j \geqslant m \end{array}  \right. \ \ \ \ (0 \leqslant i , j \leqslant m - 1)
\end{equation}
Let $\hat{\xi}$ be the cocycle on $\Gamma / \Delta$ with values in $\mathrm{Div}(\bar{C})$ given by a formula similar to (\ref{E:LocCom123}) in which $f_a$
is replaced by $\hat{a}$. It follows from the construction of $\rho$ in \S\ref{S:BrExactSeq} that $\rho(x)$ is represented by the cocycle $n \cdot \hat{\xi}$. But since
$m = \vert \Gamma / \Delta \vert$ divides $n$, the class of $n \cdot \hat{\xi}$ in $H^2(k , \mathrm{Div}(\bar{C}))$ is trivial. Thus, $\rho(x) = 0$, as required.


\medskip

(2): We will use the natural maps $$\alpha \colon \Br'(k(C)) = H^2(k , \bar{k}(C)^{\times}) \to H^2(k , \mathrm{P}(\bar{C})) \ \  \text{and} \ \  \gamma \colon H^1(k ,
\mathrm{Pic}^0(\bar{C})) \to H^2(k , \mathrm{P}(\bar{C}))$$ introduced in \S\ref{S:BrExactSeq}. Since $C(k) \neq \emptyset$, by Lemma \ref{L:Out2}, the map $\gamma$ is injective, so it is enough to show that
$$
\gamma \circ \omega_k \circ \mu = \gamma \circ \nu.
$$
By construction, $\gamma \circ \omega_k = \alpha$, so what we need to prove is that
\begin{equation}\label{E:7777}
(\alpha \circ \mu)(\chi \otimes a) = (\gamma \circ \nu)(\chi \otimes a)
\end{equation}
for all $\chi \in H^1(k , \Z/n\Z)$ and $a \in H^0(k , {}_n\mathrm{Pic}^0(\bar{C}))$. Using the same notations as in the proof of part (1), we see that
the left-hand side of (\ref{E:7777}) is represented by a cocycle $\bar{\xi}$ on $\Gamma/\Delta$ with values in $\mathrm{P}(\bar{C})$ given by a formula similar
to (\ref{E:LocCom123}) in which $f_a$ is replaced with the corresponding principal divisor $(f_a)$. On the other hand, $\nu(\chi \otimes a)$ is represented by
the following cocycle
$$
\zeta(\tau) = \chi(\tau) \cdot a \in \mathrm{Pic}^0(\bar{C}), \ \ \tau \in \Gamma
$$
(since $a \in {}_n\mathrm{Pic}^0(\bar{C})$, the product is well-defined); note that $\zeta$ descends to $\Gamma / \Delta$. We will now compute the right-hand side of
(\ref{E:7777}) using the fact that $\gamma$ is the coboundary map associated with the exact sequence (\ref{E:Picard-a}). Let $\varphi \colon \Gamma \to \Z$ be a function defined by $\varphi(\tau) = i(n/m)$ if $\tau \Delta = \sigma^i \Delta$ with $0 \leqslant i \leqslant m - 1$.
Then the function
$$
\lambda \colon \Gamma \to \mathrm{Div}^0(\bar{C}), \ \ \lambda(\tau) := \varphi(\tau) \cdot \hat{a},
$$
has the property $\pi(\lambda(\tau)) = \zeta(\tau)$, where $\pi \colon \mathrm{Div}^0(\bar{C}) \to \mathrm{Pic}^0(\bar{C})$ is the canonical map.
Thus, the right-hand side of (\ref{E:7777}) is represented by the cocycle on $\Gamma / \Delta$ given by
$$
\kappa(\sigma^i \Delta , \sigma^j \Delta) = (\varphi(\sigma^i) + \varphi(\sigma^j) - \varphi(\sigma^{i+j})) \cdot \hat{a}.
$$
But
$$
\varphi(\sigma^i) + \varphi(\sigma^j) - \varphi(\sigma^{i + j}) = \left\{ \begin{array}{ccl} 0 & , & i + j < m \\ m \cdot (n/m) = n & , & i + j \geqslant m \end{array}     \right.
$$
Since $n \cdot \hat{a} = (f_a)$, we see that $\kappa$ coincides with $\bar{\xi}$, and our claim follows.
\end{proof}

\medskip

Let $J$ be the Jacobian of $C$. As before, we will tacitly identify $\mathrm{Pic}^0(\bar{C})$ with $J(\bar{k})$ as $\Gamma$-modules. Since $n$ is prime to
$\mathrm{char}\: k$, the $n$-torsion
${}_nJ(\bar{k})$ as an abstract group is isomorphic to $\left(\Z/n\Z\right)^{2g}$, where $g = \dim J$ is the genus of $C$ (cf. \cite[5.11]{GM-Ab-Var}). Assume now that  the $n$-torsion is
$k$-{\it rational}, i.e. ${}_nJ(\bar{k})  \subset J(k)$. Then the $\cup$-product
$$
H^1(k , \Z/n\Z) \otimes H^0(k , {}_nJ) \longrightarrow H^1(k , {}_nJ)
$$
is an isomorphism. We also recall that from the long exact cohomology  sequence associated to the Kummer sequence
$$
0 \to {}_nJ \longrightarrow J \stackrel{[n]}{\longrightarrow} J \to 0,
$$
(here $[n]$ is the morphism given by multiplication by $n$) one obtains that the natural map $$\theta_k \colon H^1(k ,{}_nJ) \to {}_nH^1(k , J)$$ is always surjective, with  kernel
isomorphic to $J(k)/nJ(k)$. Combining these remarks, we obtain

\begin{lemma}\label{L:Loc-Com1}
If ${}_nJ(\bar{k}) \subset J(k)$, then $\nu$ is surjective. If in addition $J(k) = nJ(k)$, then $\theta_k$, hence $\nu$, is an  isomorphism.
\end{lemma}

\medskip

\noindent {\bf Remark \ref{S:Loc}.3.} Assume that the $n$-torsion of $J$ is $k$-rational. Then  $\nu$ is surjective, and consequently
\begin{equation}\label{E:51}
{}_n\Br(k(C))_{\mathrm{ur}} = {}_n\Br(k) + \mathrm{Im}\: \mu.
\end{equation}
(in fact, this is a direct sum). To describe $\mathrm{Im} \: \mu$ explicitly, we let $a_i$ for $i = 1, \ldots , 2g$ denote the preimage of the $i$th basic vector under a fixed
isomorphism ${}_nJ(k) \simeq (\Z/n\Z)^{2g}$, and let $f_i := f_{a_i}$ in our previous notations (in fact, one can take $f_i \in k(C)^{\times}$ to be
any function such that $(f_i) = n \cdot \hat{a}_i$ with $\hat{a}_i \in \mathrm{Div}^0(\bar{C})$ having the property that its image in $\mathrm{Pic}^0(\bar{C})$ corresponds
to $a_i$). Then the elements of $\mathrm{Im}\: \mu$ are represented by  tensor products of cyclic algebras of the form
$$
(\chi_1 , f^{m_1}_1) \otimes_K \cdots \otimes_K (\chi_{2g} ,  f^{m_{2g}}_{2g})
$$
for some characters $\chi_i \in H^1(k , \Z/n\Z)$ and some integers $0 \leqslant m_i < n$, with $i = 1, \ldots , 2g$. Now assume in addition that $k$ contains
all $n$th roots of unity, and fix an isomorphism $\Z/n\Z \simeq \mu_n$, which is equivalent to fixing a primitive root $\zeta_n
\in \mu_n(k)$. Then by Kummer theory, every $\chi \in H^1(k , \Z/n\Z)$ can be written in the form $\chi = \chi_a$ for some $a \in k^{\times}$ defined
uniquely modulo ${k^{\times}}^n$ where $\chi_a \colon \Gamma \to \mu_n$ is given by the formula
$$
\chi_a(\sigma) = \frac{\sigma\left( \sqrt[n]{a} \right)}{\sqrt[n]{a}}, \ \ \ \sigma \in \Gamma
$$
(this definition is independent of the choice of $\sqrt[n]{a}$). Let $m \vert n$ be the order of $\chi = \chi_a$.
Then the cyclic algebra $(\chi_a , f)$ is nothing but the symbol algebra $(a , f)_{m , \zeta^{n/m}_n}$ of degree $m$ constructed using the $m$th root of unity $\zeta^{n/m}_n$, which we will denote simply by $(a , f)$. Consequently, the elements of $\mathrm{Im} \:
\mu$ is represented by  tensor products of symbol algebras of the form
$$
(a_1 , f^{m_1}_1) \otimes_K \cdots \otimes_K (a_{2g} , f^{m_{2g}}_{2g})
$$
for some $a_1, \ldots , a_{2g} \in k^{\times}$ (with the same $m_i$  as above). Together with (\ref{E:51}), this gives an explicit description
of ${}_n\Br(k(C))_{\mathrm{ur}}$. Here is one concrete example.

Let $C$ be an elliptic curve over a (perfect) field $k$ of characteristic $\neq 2, 3$ with $k$-rational 2-torsion. Then $C$ can be given by a Weierstrass
equation $$y^2 = (x - a)(x - b)(x - c)$$ for some pairwise distinct $a, b, c \in k$. Set $K = k(C)$. In this case, the genus $g = 1$, and one can take
$f_1 = x - a$ and $f_2 = x - b$ in the above notations. So, the preceding discussion leads to the following result
\cite[Theorem 3.6]{CGu}: {\it ${}_2\Br(K)_{\mathrm{ur}} = {}_2\Br(k) + I$ where the subgroup $I$ consists of classes of  bi-quaternionic algebras of the form
$(r , x-a) \otimes_K (s , x-b)$ for some $r , s \in k^{\times}$.} (Here $(\alpha , \beta)$ stands for the quaternion algebra over $K$ corresponding to the pair $\alpha ,
\beta \in K^{\times}$.) This result was used in \cite{CRR2} to prove Theorem 1 for $K = k(C)$, where $C$ is an elliptic
curve over a number field $k$, with an explicit estimation. The argument in \cite{CRR2} can be viewed as a prototype of our proof of Theorem~1 in the general case.



\addtocounter{thm}{2}

\medskip

\noindent {\bf \ref{S:Loc}.4. The case of a strictly henselian field.} Let $\ell$ be a field equipped with a discrete valuation $v$. Throughout
this subsection we will assume that {\it $\ell$ is henselian and the residue field $\ell^{(v)}$ is separably closed} (i.e., $\ell$ is {\it strictly
henselian}). We will later apply
the results proved here to the case where $\ell$ is the maximal unramified extension of a complete discretely valued field. Fix an integer
$n > 1$ prime to $\mathrm{char}\: \ell^{(v)}$. Let $C$ be a smooth projective curve over $\ell$ having good reduction at $v$. Thus, there exists a smooth projective curve $\mathcal{C}$ over the valuation ring $\mathcal{O}_{\ell} \subset \ell$ with generic fiber
$\mathcal{C} \times_{\mathcal{O}_{\ell}} \ell \simeq C$. We let $\underline{C}^{(v)} = \mathcal{C} \times_{\mathcal{O}_{\ell}}
\ell^{(v)}$ denote the corresponding closed fiber or reduction (assumed to be irreducible). Let $J$ be the Jacobian of $C$.
In this situation, the following properties hold, enabling us to apply our previous results.

\medskip

(i) $C(\ell) \neq \emptyset$. This immediately follows from Hensel's lemma since $\ell^{(v)}$ is separably closed and therefore
$\underline{C}^{(v)}(\ell^{(v)}) \neq \emptyset$ (cf. \cite[Theorem 3.5.50]{Poon}).

\smallskip

(ii) ${}_n\Br(\ell) = 0$ - cf., for example, \cite[Ex. 3, p. 187]{Serre-LF}. Consequently, $\omega_{\ell} \colon {}_n\Br(\ell(C))_{\mathrm{ur}} \to {}_nH^1(\ell , J)$ is an isomorphism.

\smallskip

(iii) $n \cdot J(\ell) = J(\ell)$,  hence $\theta_{\ell}$ is an isomorphism (cf. the discussion prior to Lemma \ref{L:Loc-Com1}).
Indeed, there exists an abelian scheme $\mathcal{J}$ over $\mathcal{O}_{\ell}$
with generic fiber $\mathcal{J} \times_{\mathcal{O}_{\ell}} \ell = J$ (cf. \cite[Ch. 9]{BBL}). Since $\mathcal{J}$ is proper (and separated) over $\mathrm{Spec} \: \mathcal{O}_{\ell}$, every $\ell$-point $\mathrm{Spec} \: \ell \to \mathcal{J}$ uniquely extends to an $\mathcal{O}_{\ell}$-point $\mathrm{Spec} \: \mathcal{O}_{\ell} \to \mathcal{J}$ by the valuative criterion. Thus, $$\mathcal{J}(\mathcal{O}_{\ell}) = \mathcal{J}(\ell) = J(\ell).$$ Let $\underline{J}^{(v)} = \mathcal{J}
\times_{\mathcal{O}_{\ell}} \ell^{(v)}$ denote the reduction (which is the Jacobian of $\underline{C}^{(v)}$). Since $n$ is prime to $\mathrm{char} \: \ell^{(v)}$, the
multiplication by $n$ map $[n] \colon \underline{J}^{(v)} \to \underline{J}^{(v)}$ is \'etale, so $n \cdot \underline{J}^{(v)}(\ell^{(v)}) = \underline{J}^{(v)}(\ell^{(v)})$ as $\ell^{(v)}$ is separably closed. Then by Hensel's lemma $n \cdot \mathcal{J}(\mathcal{O}_{\ell}) = \mathcal{J}(\mathcal{O}_{\ell})$, and our claim follows.

\smallskip

(iv) In the notations of (iii), the restriction  of the reduction map $\mathcal{J}(\mathcal{O}_{\ell}) \to \underline{J}^{(v)}(\ell^{(v)})$ to ${}_nJ(\ell)$
is injective. Indeed, the kernel of the reduction map has no $n$-torsion - see \cite[C.2]{HinSilv} for a proof using formal groups. For a different proof one observes that the reduction map ${}_n\mathcal{J}(\mathcal{O}_{\ell}) \to {}_n\underline{J}^{(v)}(\ell^{(v)})$ is surjective by Hensel's lemma, so it must also be injective as both have
the same order $n^{2g}$ where $g$ is the genus of $C$ (cf. \cite[Exercise C.9]{HinSilv}).

\medskip

In view of propeties (i)-(iii), Proposition \ref{P:Loc-Com1} and Lemma \ref{L:Loc-Com1} with $k = \ell$ yield the following.
\begin{prop}\label{P:Loc-Com2}
$\mu$ is a bijection.
\end{prop}


Since $C$ has good reduction at $v$, there exists a canonical (unramified) extension $\dot{v}$ of $v$ to $\ell(C)$  with residue field $\mathcal{L} := \ell^{(v)}(\underline{C}^{(v)})$. Let
$$
\rho_{\dot{v}}\colon {}_n\Br(\ell(C)) \to H^1(\mathcal{L} , \Z/n\Z)
$$
be the corresponding residue map. We will now combine the previous results with property (iv) to establish the following.
\begin{thm}\label{T:Inj777}
The restriction of $\rho_{\dot{v}}$ to ${}_n\Br(\ell(C))_{\mathrm{ur}}$ is injective.
\end{thm}
\begin{proof}
By Hensel's lemma, the field $\ell$ contains a primitive $n$th root of unity $\zeta$, which we will fix and  use in the sequel to identify $\mu_n$
with $\Z/n\Z$ and  construct symbol algebras.
Note that the  group $\ell^{\times} / {\ell^{\times}}^n$ is cyclic of order $n$ with generator
$\pi {\ell^{\times}}^n$, where $\pi \in \ell$ is an arbitrary uniformizer, and  by Proposition \ref{P:Loc-Com2} the map $\mu$ is a bijection. So, the description
given in Remark 6.3 tells us that every element of ${}_n\Br(\ell(C))_{\mathrm{ur}}$ is represented by a symbol algebra $(\pi , f_a)$ of degree $n$
for a unique $a \in {}_nJ(\ell)$. As the extension $\dot{v} \vert v$ is unramified,  $\pi$ remains a uniformizer for $\dot{v}$. Let us identify
$H^1(\mathcal{L} , \Z/n\Z)$ with $H^1(\mathcal{L} , \mu_n) \simeq \mathcal{L}^{\times}/{\mathcal{L}^{\times}}^n$ using the image of $\zeta$ in $\ell^{(v)}$ (which is
still a~primitive $n$th root of $1$). We then have the following formula for the residue (cf. \cite[Example 7.1.5]{GiSz}):
\begin{equation}\label{E:Residue}
\rho_{\dot{v}}([\pi , f_a]) = (-1)^{\dot{v}(f_a)} \overline{\pi^{-\dot{v}(f_a)} f_a} (\mathcal{L}^{\times})^n \in \mathcal{L}^{\times}/{\mathcal{L}^{\times}}^n,
\end{equation}
where for a function $h \in \ell(C)$ such that $\dot{v}(h) = 0$ we let $\bar{h}$ denote the image of $h$ in $\mathcal{L}^{\times}$. Let $g := \pi^{-\dot{v}(f_a)} f_a$. Then
$(g) = (f_a) = n \cdot \hat{a}$, where $\hat{a} \in \mathrm{Div}(C)$ is a divisor whose image in $J(\ell)$ is $a$ and $(\bar{g}) = n \cdot \hat{b}$ where $\hat{b} \in \mathrm{Div}(\underline{C}^{(v)})$ is a divisor whose image in $\underline{J}^{(v)}(\ell^{(v)})$ is $b = \bar{a}$ (the image of $a$ under the reduction map).

Now, if the residue of $[\pi , f_a]$ is trivial, it follows from (\ref{E:Residue}) that $(\bar{g}) \in n \cdot \mathrm{P}(\underline{C}^{(v)})$. This means that $\hat{b} \in
\mathrm{P}(\underline{C}^{(v)})$, and therefore $b = \bar{a} = 0$. Invoking property (iv), we conclude that $a = 0$, and therefore $[\pi , f_a]$ is trivial, as required. \end{proof}

\bigskip

\section{The finiteness of $\omega_k({}_n\Br(k(C))_V)$}\label{S:FiniteA}

We begin with the following immediate consequence of Theorem \ref{T:Inj777}.
As before, let $$\theta_k \colon H^1(k , {}_nJ) \to {}_nH^1(k , J)$$ be the natural (surjective) map given by the inclusion ${}_nJ \hookrightarrow J$, and $\omega_k \colon
\Br'(k(C))_{\mathrm{ur}} \to H^1(k , J)/\Phi(C , k)$ be the map constructed in \S\ref{S:BrExactSeq}.
\begin{prop}\label{P:F-omega1}
Let $C$ be a geometrically irreducible smooth projective curve over a field $k$ having a $k$-rational point (hence the group $\Phi(C ,k)$ is trivial) .
Let $v$ be a discrete valuation of $k$ such that $\mathrm{char}\: k^{(v)}$ is prime to $n$ and
$C$ has good reduction at $v$, and denote by  $\dot{v}$  the canonical extension of $v$ to $k(C)$. If $x \in {}_n\Br(k(C))_{\mathrm{ur}}$
is unramified at $\dot{v}$, then
$$
\theta^{-1}_k(\omega_k(x)) \subset H^1(k , {}_nJ)_{\{v\}}.
$$
\end{prop}
\begin{proof}
Let $\ell = k^{\mathrm{ur}}_v$ be the maximal unramified extension of the completion $k_v$. First note that since  $C(k) \neq \emptyset$, the groups $\Phi(C , k)$ and $\Phi(C ,
\ell)$ are trivial by Lemma \ref{L:Out2}, and we have the following commutative diagram
$$
\xymatrix{{}_n\Br(k(C))_{\mathrm{ur}} \ar[d]_{\alpha} \ar[rr]^(.50){\omega_k} & & {}_nH^1(k , J) \ar[d]^{\beta} \\ {}_n\Br(\ell(C))_{\mathrm{ur}} \ar[rr]^(.50){\omega_{\ell}} & & {}_nH^1(\ell , J_{\ell})}
$$
where $J_{\ell} := J \times_k \ell$ is the Jacobian of $C_{\ell} := C \times_k \ell$, $\alpha$ is the natural map, and $\beta$ is given by restriction.
Pick any $y \in H^1(k , {}_nJ)$ such that $\theta_k(y) = \omega_k(x)$. Since $C_{\ell}$ has good reduction and $x$ is unramified at $\dot{v}$, we conclude from Theorem \ref{T:Inj777} that $\alpha(x) = 0$, and therefore $\beta(\omega_k(x)) = 0$. On the other hand, from the commutative diagram
$$
\xymatrix{H^1(k , {}_nJ) \ar[rr]^{\theta_k} \ar[d]_{\gamma} & & {}_nH^1(k , J)  \ar[d]^{\beta} \\ H^1(\ell , {}_nJ) \ar[rr]^{\theta_{\ell}} & & {}_nH^1(\ell , J)}
$$
we see that $\theta_{\ell}(\gamma(y))= 0$. Since $\theta_{\ell}$
is an isomorphism (see \ref{S:Loc}.5, (iii)), we conclude that
$\gamma(y) = 0$ and hence $y$ is unramified at $v$ (cf. \S \ref{S:UnramClass}).
\end{proof}

\medskip

\noindent {\bf \ref{S:FiniteA}.2. A finiteness result.} Let $C$ be a  geometrically connected smooth projective curve over a field $k$, and
let $n > 1$ be an integer prime to $\mathrm{char}\: k$. Denote by $V_0$  the set of geometric places of $k(C)$. Furthermore, suppose we are given
a set $V_1$ of discrete valuations of $k$ such that for each $v \in V_1$ the characteristic of the residue field $k^{(v)}$ is prime to $n$ and $C$ has good reduction
at $v$. We then let $\dot{v}$ denote the canonical extension of $v$ to $k(C)$, and set
$$
\dot{V}_1 = \{ \dot{v} \, \vert \, v \in V_1 \} \ \ \text{and} \ \ V = V_0 \cup \dot{V}_1.
$$
Finally, fix a finite Galois extension $\ell/k$ over which $C$ has a rational point, and let $V^{\ell}_1$ denote the set of all extensions of the valuations in $V_1$ to $\ell$.
\begin{thm}\label{T:omega2}
{\rm (1)} If the group $H^1(\ell , {}_nJ)_{V^{\ell}_1}$ is finite, then so is the group  $\omega_k({}_n\Br(k(C))_V)$ and
$$
\vert \omega_k({}_n\Br(k(C))_V) \vert \ \ \ \text{divides} \ \ \  \vert H^1(\ell/k , {}_nJ(\ell)) \vert \cdot \vert \Phi(C , k) \vert  \cdot \vert H^1(\ell , {}_nJ)_{V^{\ell}_1} \vert.
$$

\smallskip

\noindent {\rm (2)} If every $v \in V_1$ is unramified in $\ell/k$ and the group $H^1(k , {}_nJ)_{V_1}$ is finite, then
$$
\vert \omega_k({}_n\Br(k(C))_V) \vert \ \ \ \text{divides} \ \ \  \vert \Phi(C , k) \vert \cdot \vert H^1(k , {}_nJ)_{V_1} \vert.
$$
\end{thm}
\begin{proof}
Applying Lemma \ref{L:RC555} to the canonical homomorphism $H^1(k , J) \stackrel{\phi_k}{\longrightarrow} H^1(k , J)/\Phi(C , k)$ and taking into account
the surjectivity of $\theta_k$, we see that for the composition $\theta'_k = \phi_k \circ \theta_k$, the index
$$
\left[\,{}_n\left(H^1(k , J)/\Phi(C , k) \right) \, : \, {\rm Im}\: \theta'_k  \,\right]
$$
divides $\vert \Phi(C , k) \vert$. So, if we set
$$
X = (\theta'_k)^{-1}\left(\omega_k({}_n\Br(k(C))_V)\right),
$$
then the index $[\omega_k({}_n\Br(k(C))_V) :  \theta'_k(X)]$ divides $\vert \Phi(C , k) \vert$. In particular, if $X$ is finite
then $\omega_k({}_n\Br(k(C))_V)$ is also finite, of order dividing $\vert X \vert \cdot
\vert \Phi(C , k) \vert$.

Now, consider the following commutative diagram
$$
\xymatrix{\Br'(k(C))_{\mathrm{ur}} \ar[d]_{\alpha} \ar[rr]^{\omega_k} & & H^1(k , J)/\Phi(C,k) \ar[d]^{\beta'} \\ \Br'(\ell(C))_{\mathrm{ur}} \ar[rr]^{\omega_{\ell}} & & H^1(\ell , J)/\Phi(C , \ell)}
$$
(note that $\Phi(C , \ell)$ is the trivial group since $C(\ell) \neq \emptyset$) and let $\lambda \colon H^1(k , {}_nJ) \to H^1(\ell , {}_nJ)$ be the restriction map. We then have the following inclusions
$$
\lambda(X) \subset \theta_{\ell}^{-1}(\omega_{\ell}(\alpha({}_n\Br(k(C))_V))) \subset H^1(\ell , {}_nJ)_{V_1^{\ell}}
$$
(observe that $\theta'_{\ell} = \theta_{\ell}$ in the above notations). The first inclusion follows from the definitions, and the second is a consequence of Proposition
\ref{P:F-omega1} as the elements of $\alpha({}_n\Br(k(C))_V)$ are unramified at $V_1^{\ell}$.
Since $\mathrm{Ker}\: \lambda = H^1(\ell/k , {}_nJ(\ell))$, part (1) follows immediately. If $v \in V_1$ is unramified in $\ell$ then for $w \vert v$ we have $\ell^{\mathrm{ur}}_w = k^{\mathrm{ur}}_v$. Then the fact that $\lambda(X) \subset H^1(\ell , {}_nJ)_{V^{\ell}_1}$ implies that $X \subset H^1(k , {}_nJ)_{V_1}$, yielding part (2).
\end{proof}

\bigskip

\section{Proof of Theorem 2}\label{S:PfT2}

Let $K$ be a finitely generated field, and let $n > 1$ be an integer prime to $\mathrm{char}\: K$. If $K$ is a global field then our assertion is well-known
(cf. Lemma \ref{L:RC444}).
So, we will assume in this section that $K$ is \emph{not} global, and pick for $K$ a presentation $K = k(x , y)$, where $k = P(t_1, \ldots , t_r)$ is a purely transcendental
extension of a global field $P$, with the properties listed in Proposition \ref{P:FGF1}. Let $C$ be a geometrically irreducible smooth projective $k$-defined curve with  function field $k(C) = k(x , y) = K$, and let $V_0$ be the set of geometric places of $k(C)$. Furthermore, the construction described in \S\ref{S:FinGenF}.2 yields a~set of discrete valuations $T \subset V^P$ with finite complement and the corresponding sets $V(T)$ and  $\widetilde{V(T)} = \{ \tilde{v} \: \vert \: v \in V(T) \}$ of places of $k$ and $K$, respectively. In addition, by reducing $T$ if necessary  we may assume that the ring of $S$-integers $\mathcal{O}_P(S)$ for $S = V^P \setminus T$ is a UFD. Set
$$
V = V_0 \cup \widetilde{V(T)}.
$$
Our goal is to show that the unramified Brauer group ${}_n\Br(K)_V$ is finite. As we already observed in \S \ref{S:BrExactSeq}, it is enough to check conditions (I) and (II) of  Proposition \ref{P:Out1}, i.e. the finiteness of $\iota^{-1}_k({}_n\Br(K)_V)$ and $\omega_k({}_n\Br(K)_V)$, respectively, where $\iota_k$ and $\omega_k$ are the maps from the standard exact sequence (\ref{E:Out1a}). The fact that (I) holds is established in Theorem \ref{T:FinGenF1}, so we only need to verify (II). The fact that $V(T)$ satisfies condition (A) of
\S\ref{S:Intro} implies that there exists a finite subset $S_1 \subset V(T)$ such that $C$ admits a smooth proper model $\mathscr{C} \to \mathrm{Spec}\: \mathscr{A}$ where
$$
\mathscr{A} = \bigcap_{v \in V(T) \setminus S_1} \mathcal{O}_{k , v}
$$
(cf. \cite[Prop. A.9.1.6]{HinSilv}). For $v \in V(T) \setminus S_1
$, the base change $\mathcal{C}_v := \mathscr{C} \times_{\mathscr{A}} \mathcal{O}_{k , v}$ is a smooth model over $\mathcal{O}_{k , v}$, hence defines
a canonical (unramified) extension $\dot{v}$ of $v$ to $K$.
\begin{lemma}\label{L:PfT2-1}
There exists a finite subset $S_2 \subset V(T) \setminus S_1$ such that for $v \in V(T) \setminus (S_1 \cup S_2)$ we have $\tilde{v} = \dot{v}$.
\end{lemma}
\begin{proof}
It follows from Proposition \ref{P:RC777} and Lemma \ref{L:RC3} that there exists a finite subset $S' \subset V(T) \setminus S_1$ such that for every
$v  \in V(T) \setminus (S_1 \cup S')$, the valuation $\hat{v}$ of $k(x)$ has a \emph{unique} extension $\tilde{v}$ to $K$. On the other hand,  there exists
a finite subset $S'' \subset V(T)$ such that $\dot{v}(x) = 0$ for all $v \in V(T) \setminus (S_1 \cup S'')$. Set $S_2 = S' \cup S''$, and let $v \in V(T)
\setminus (S_1 \cup S_2)$. In view of the uniqueness of an extension of $\hat{v}$ from $k(x)$ to $K$, it is enough to show that the restriction $\dot{v} \vert
k(x)$ coincides with $\hat{v}$. But this follows from the uniqueness of an  extension $w$ of $v$ to $k(x)$ such that $w(x) = 0$ (cf. \cite[Ch. 6, \S10, n$^o$ 1, Prop.~1]{Bour}).
\end{proof}

\medskip

Let $V_1 = V(T) \setminus (S_1 \cup S_2 \cup S_3)$, where $S_3 = \{ v \in V(T) \: \vert \: v(n) \neq 0 \}$, and let $\widetilde{V}_1 = \{ \tilde{v} \: \vert \: v \in V_1 \}$
and $V' = V_0 \cup \widetilde{V}_1$. Then of course $\omega_k({}_n\Br(K)_V) \subset \omega_k({}_n\Br(K)_{V'})$. On the other hand, in view of Lemma~\ref{L:PfT2-1}
and our constructions, we can use Theorem \ref{T:omega2} to conclude that the finiteness of $\omega_k({}_n\Br(K)_{V'})$ would follow from the finiteness of
$H^1(\ell , {}_nJ)_{V_1^{\ell}}$, where $J$ is the Jacobian of $C$, for some finite Galois extension $\ell/k$ such that $C(\ell) \neq \emptyset$, with $V_1^{\ell}$ being
the set of all extensions of the valuations from $V_1$ to $\ell$. One can find a finite Galois extension $\ell$ of $k$ so that $C(\ell)
\neq \emptyset$ and the $n$-torsion of $J$ is $\ell$-rational\footnote{Recall that the existence of the Weil pairing on $J$ shows that this condition implies that $\mu_n \subset \ell$, cf., for example, \cite[Exercise A.7.8]{HinSilv}.
}. Then according to Proposition \ref{P:UR778}, to prove the finiteness of $H^1(\ell , {}_nJ)_{V_1^{\ell}}$, it is enough
to check condition $(\mathbf{F}(\ell , V_1^{\ell}))$ of \S\ref{S:UnramClass},  i.e. the finite generation of the groups $\mathrm{U}(\ell , V_1^{\ell})$ and $\mathrm{Pic}(\ell ,
V_1^{\ell})$. Let
$$
\mathscr{B} = \bigcap_{v \in V_1} \mathcal{O}_{k , v} \ \ \ \text{and} \ \ \ \mathscr{B}_{\ell} = \bigcap_{w \in V_1^{\ell}} \mathcal{O}_{\ell , w}.
$$
Then $\mathscr{B}_{\ell}$ is the integral closure of $\mathscr{B}$ in $\ell$. By our construction, the ring of $S$-integers $\mathcal{O}_P(S)$  for $S = V^P \setminus T$
is a UFD, which implies that the ring
$\mathscr{B}$ is the localization of the polynomial ring $\mathcal{O}_P(S)[t_1, \ldots , t_r]$ with respect to a multiplicative set generated by
a finite set, hence a finitely generated ring. It follows that $\mathscr{B}_{\ell}$ is a finitely generated $\mathscr{B}$-module (cf. \cite[Ch. 5, \S1, n$^o$ 6, Cor. 1]{Bour}), and consequently also a finitely generated integral domain. Then the finite generation of $\mathrm{U}(\ell , V_1^{\ell}) = \mathscr{B}^{\times}_{\ell}$ is a classical result of P.~Samuel \cite[Th\'eor\`eme 1]{Sa}. Furthermore, let $X = \mathrm{Spec}\: \mathscr{B}$ and $X_{\ell} = \mathrm{Spec}\: \mathscr{B}_{\ell}$. Clearly, $X$ is regular in codimension 1, so $X_{\ell}$, being the normalization of $X$ in $\ell$, is also regular in codimension 1. Then $\mathrm{Pic}(\ell , V_1^{\ell})$ coincides with the divisor class $\mathrm{Cl}(X_{\ell})$,  which is finitely generated by \cite[Th\'eor\`eme 1]{Kahn}. This completes the proof of Theorem 2.  \hfill $\Box$

\medskip

\addtocounter{thm}{1}

\noindent {\bf \ref{S:PfT2}.2. Explicit estimates.} The proof of Theorem 2 given above enables one to give explicit bounds on the order of the $n$-torsion of the unramified Brauer
group, hence on the size of the genus of a division algebra of degree $n$. To illustrate the method, we will develop some explicit formulas in the case $K = k(C)$ where $C$ is a geometrically irreducible smooth projective curve over a number field $k$ such that $C(k) \neq \emptyset$; numerical examples will be given in \S \ref{S:Examples}. As above, let $V_0$ be the set of geometric places of $K$. Fix a finite subset $S \subset V^k$ that contains $V_{\infty}^k$, all divisors of $n$ and also all $v \in V^k \setminus V_{\infty}^k$
where $C$ does not have good reduction. Each $v \in T:= V^k \setminus S$ canonically extends to a valuation $\dot{v}$ of $K$ (this extension is determined by the given smooth $\mathcal{O}_{k , v}$-model of $C$), and we set
$$
\dot{T} = \{ \dot{v} \: \vert \: v \in T \} \ \ \ \text{and} \ \ \ V = V_0 \cup \dot{T}.
$$
We will now estimate the order of the unramified Brauer group ${}_n\Br(K)_V$. First, by Proposition \ref{P:Out1}, the order $\vert {}_n\Br(K)_V \vert$ divides
$\vert \iota_k^{-1}({}_n\Br(K)_V) \vert \cdot \vert \omega_k({}_n\Br(K)_V) \vert$. Next, it follows from Corollary \ref{C:RC999} that the first factor $\vert
\iota_k^{-1}({}_n\Br(K)_V) \vert$ divides the number $\beta(n, k, T)$ introduced in Lemma \ref{L:RC444}, and consequently divides $(n , 2)^an^b$ where $a = \vert
V_{\mathrm{real}}^k \vert$ and $b = \vert S \setminus V_{\infty}^k \vert$. On the other hand, since $C(k) \neq \emptyset$, we can apply Theorem \ref{T:omega2} with $\ell = k$ to
conclude that $\vert \omega_k({}_n\Br(K)_V) \vert$ divides $\vert H^1(k , {}_nJ)_{T} \vert$ (we observe that Theorem \ref{T:omega2} enables one to deal also with the situation where $C(k) = \emptyset$ but then the equations become more cumbersome). To estimate the order of $H^1(k , {}_nJ)_{T}$, we need to pick a finite Galois extension $\ell/k$ that contains the $n$-torsion of $J$, hence also contains $\mu_n$. Then applying Proposition \ref{P:UR778} with $M = {}_nJ \simeq (\Z/n\Z)^{2g}$, where $g$  is the genus of $C$, we obtain from (\ref{E:UR888}) that $\vert H^1(k , {}_nJ)_T \vert$ divides
$$
\vert H^1(\ell/k , {}_nJ(\ell)) \vert \cdot  \left( \vert \mathrm{U}(\ell , T^{\ell})/\mathrm{U}(\ell , T^{\ell})^n \vert \cdot \vert {}_n\mathrm{Pic}(\ell , T^{\ell}) \vert \right)^{2g}.
$$
(We let $S^{\ell}$ (resp., $T^{\ell}$) denote the set of all extensions of valuations in $S$ (resp., $T$) to $\ell$; note that $T^{\ell} = V^{\ell} \setminus S^{\ell}$). As we indicated in \S5.5,
$$
\vert \mathrm{U}(\ell , T^{\ell})/\mathrm{U}(\ell , T^{\ell})^n \vert = (n , \vert \mu(\ell) \vert) \cdot n^{\vert S^{\ell} \vert - 1},
$$
hence divides $n^{\vert S^{\ell} \vert}$, and $\vert {}_n \mathrm{Pic}(\ell , T^{\ell}) \vert$
divides $h_{\ell}(S^{\ell})$, the class number number of the ring $\mathcal{O}_{\ell}(S^{\ell})$ of $S^{\ell}$-integers in $\ell$ (in fact, it divides $h_{\ell}(S^{\ell} , n)$
as introduced in \S\ref{S:UnramClass}.5). Finally, one can estimate the order of $H^1(\ell/k , {}_nJ(\ell))$ as indicated in \S\ref{S:UnramClass}.5. To give a really ``clean'' sample result however, let us consider the situation where one can take $\ell = k$.
\begin{thm}\label{T:ExplEst}
Keep the above notations and assume that $C(k) \neq \emptyset$ and  $J$ has $k$-rational $n$-torsion. Then the order of ${}_n\Br(K)_V$ divides
$$
(n , 2)^a \cdot n^{b + 2g \vert S \vert} \cdot h_k(S)^{2g},
$$
and consequently divides
$$
n^{(2g + 1)\vert S \vert} \cdot h_k(S)^{2g}.
$$
\end{thm}

Combining this result with  (\ref{E:genus1}), we obtain the following.

\begin{cor}\label{C:ExplEst}
In the above notations, given a central division algebra $D$ of degree $n$ over $K$, we have the following estimation of the size of its
genus:
$$
\vert \gen(D) \vert \leqslant \varphi(n)^r \cdot n^{(2g + 1)\vert S \vert} \cdot h_k(S)^{2g},
$$
where $r$ is the number of $v \in V$ where $D$ ramifies and $\varphi$ is the Euler function.
\end{cor}

\medskip

\section{Examples}\label{S:Examples}

\noindent {\bf \ref{S:Examples}.1. Elliptic curves and $n = p$ a prime.} In this subsection, we will deviate from our standard notation $C$ for a given curve,
and will use $E$ to denote an elliptic curve over a number field $k$. Mimicking the construction described in \S\ref{S:PfT2}.2, we pick a finite subset $S \subset V^k$
that contains $V_{\infty}^k$, all extensions of the $p$-adic valuation of $\Q$, and all places of bad reduction -- these can be easily determined from the
Weierstrass equation of $E$. Then each $v \in T := V^k \setminus S$ extends canonically to a valuation $\dot{v}$ on  $K = k(E)$, and we set
$$
\dot{T} = \{ \dot{v} \: \vert \: v \in  T \} \ \ \ \text{and} \ \ \ V = V_0 \cup \dot{T},
$$
where $V_0$ is the set of geometric places of $K$. Again, $\vert {}_p\Br(K)_V \vert$ divides $\vert \iota^{-1}_k({}_p\Br(K)_V) \vert \cdot \vert \omega_k({}_p\Br(K)_V) \vert$
in our standard notations. As in \S\ref{S:PfT2}.2, the first factor divides $(p , 2)^a \cdot p^b$ where $a = \vert V_{\mathrm{real}}^k \vert$ and $b = \vert S \setminus V_{\infty}^k \vert$. Let $\ell$ be the (Galois) extension of $k$ generated by the coordinates of elements of order $p$ in $J = E$. Then as in \S\ref{S:PfT2}.2, the second factor
divides
\begin{equation}\label{E:Examples1}
\vert H^1(\ell/k , {}_pE) \vert \cdot \vert \mathrm{U}(\ell , T^{\ell})/\mathrm{U}(\ell , T^{\ell})^p  \vert^2 \cdot \vert {}_p\mathrm{Pic}(\ell , T^{\ell}) \vert^2,
\end{equation}
where $T^{\ell} = V^{\ell} \setminus S^{\ell}$ and $S^{\ell}$ consists of all extensions of the valuations in $S$ to $\ell$.  We know from \S\ref{S:PfT2}.2 that the product of the second and the third factors in (\ref{E:Examples1}) divides $p^{2\vert S^{\ell} \vert} \cdot h_{\ell}(S^{\ell})^2$ where $h_{\ell}(S^{\ell})$ is the class number of the ring
$\mathcal{O}_{\ell}(S^{\ell})$ of $S^{\ell}$-integers in $\ell$. To estimate the first factor, we need the following simple computation, where we write $\mathbb{F}_p$ for $\Z/p\Z$.
\begin{lemma}\label{L:p-coh}
Let $M = \mathbb{F}_p^2$, and let $G$ be any subgroup of $\mathrm{Aut}(M) = \mathrm{GL}_2(\mathbb{F}_p)$. Then $H^1(G , M)$ has order
either $1$ or $p$; the order is $1$ if $p = 2$.
\end{lemma}
\begin{proof}
Let $G_p$ be a Sylow $p$-subgroup of $G$. Since $pM = 0$, the restriction $H^1(G , M) \to H^1(G_p , M)$ is injective. If $G_p$ is trivial then there is
nothing to prove; otherwise $G_p$ is conjugate in $\mathrm{GL}_2(\mathbb{F}_p)$ to $U = \langle u \rangle$ where $u = \left( \begin{array}{cc} 1 & 1 \\ 0 & 1 \end{array} \right)$,
and it is enough to compute $H^1(U , M)$. Since $U$ is cyclic, $H^1(U , M) \simeq \mathrm{Ker}\: N / (u - 1)M$, where $N = 1 + u + \cdots + u^{p-1} \in \mathrm{End}(M)$ is the norm map. A direct computation shows that $N = 0$ if $p > 2$, and is given by the matrix $\left(\begin{array}{cc} 0 & 1 \\ 0 & 0 \end{array} \right)$ if $p = 2$. Thus, $\mathrm{Ker}\: N$ is $\mathbb{F}_p$ in the first case, and $\mathbb{F}_p e_1$ in the second (here $e_1 , e_2$ is the standard basis of $\mathbb{F}_p^2$). On the other hand, $(u - 1)M = \mathbb{F}_p e_1$ in all cases. We see that $\vert H^1(U , M) \vert = p$ if $p > 2$, and is $1$ if $p = 2$, so our assertion follows.
\end{proof}
In our situation, ${}_pE \simeq \mathbb{F}^2_p$ and $\Ga(\ell/k)$ embeds in $\mathrm{GL}_2(\mathbb{F}_p)$, so the lemma applies. Putting everything together, we obtain the following.
\begin{prop}\label{P:Examples1}
The order of ${}_p\Br(K)_V$ divides
$$
\left\{\begin{array}{lcl} 2^{a + b + 2 \vert S^{\ell} \vert} \cdot h_{\ell}(S^{\ell})^2 & \text{if} & p = 2, \\
p^{1 + b + 2\vert S^{\ell} \vert} \cdot h_{\ell}(S^{\ell})^2 & \text{if} & p > 2. \end{array}   \right.
$$
\end{prop}
(Note that $a + b = \vert S \vert - c$, where $c$ is the number of complex places of $k$.)

\medskip

Here is one explicit example for $k = \Q$ and $p = 3$.

\noindent {\sc Example \ref{S:Examples}.3.} In \cite[Theorem 2.2]{Palad}, L.~Paladino constructed a 2-parameter family $\mathcal{F}_{\beta , h}$ $(\beta , h \in \Q^{\times})$ of elliptic curves over $\Q$ having a Weierstrass equation of the form $y^2 = x^3 + b_{\beta , h} x + c_{\beta , h}$ with
$$
b_{\beta , h} = -27 \frac{\beta^4}{h^4} + 18 \frac{\beta^3}{h^2} - 9 \frac{\beta^2}{2} + 3 \frac{\beta h^2}{2} - 3 \frac{h^4}{16}
 \ \ \ \text{and} \ \ \ c_{\beta , h} = 54 \frac{\beta^6}{h^6} - 54 \frac{\beta^5}{h^4} + 45 \frac{\beta^4}{2 h^2} - 15 \frac{\beta^2 h^2}{8} - 3 \frac{\beta h^4}{8} -
 \frac{1}{32h^6},
$$
and  discriminant
$$
\Delta = - \frac{216\beta^3(h^4 - 6\beta^2h^2 + 12\beta^3)}{h^6}.
$$
The family $\mathcal{F}_{\beta , h}$ contains infinitely many non-isomorphic curves, and for any $E \in \mathcal{F}_{\beta , h}$ the field $\ell$ generated by the 3-torsion
of $E$ is $\Q(\zeta_3)$ where $\zeta_3$ is a primitive 3rd root of unity, hence has class number one. It follows now from Proposition \ref{P:Examples1}  that for any choice of $T \subset V^{\Q}$ as above, the order of ${}_3\Br(\Q(E))_V$ divides $3^{\vert S \vert + 2\vert S^{\ell} \vert}$.

In particular, taking $\beta = h = 1$ yields the curve $E$ with the Weierstrass equation
$$
y^2 = x^3 - \frac{195}{16} x + \frac{647}{32}
$$
and discriminant $\Delta =  - 1512 = - 2^3 \cdot 3^3$. It follows that one can take $S = \{\infty, 2, 3 \}$. Every prime in $S$ remains prime in $\ell$, so
$\vert S^{\ell} \vert = 3$. Thus, the order of ${}_3\Br(\Q(E))_V$ divides $3^9$. Combining this with (\ref{E:genus1}), we obtain the following bound on the size
of the genus of a central cubic division algebra $D$ over $K = \Q(E)$:
$$
\vert \gen(D) \vert \leqslant 2^r \cdot 3^9,
$$
where $r$ is the number of $v \in V$ where $D$ ramifies.

We note that the structure of the field generated by the 3-torsion is described explicitly in \cite{BanPal} for all elliptic curves over $\Q$.

\medskip

\noindent {\bf \ref{S:Examples}.4. Hyperelliptic curves and $n = 2$.} Let $C$ be a hyperelliptic curve over a number field $k$, i.e. a smooth projective geometrically connected
curve of genus $g \geqslant 1$ admitting a finite $k$-morphism $C \to \mathbb{P}^1$ of degree 2. It is well-known that $C$ contains  the  affine plane curve given by an equation of the form $y^2 = f(x)$ where $f(x) \in k[x]$ is a separable polynomial of degree $m \geqslant 3$. Then $g$ and $m$ are related by the formula
$$
g = \left\{ \begin{array}{ll}  (m - 1)/2, & m \ \ \text{even}, \\ (m-2)/2, & m \ \ \text{odd},   \end{array}   \right.
$$
(cf. \cite[Prop. 7.4.24]{Liu}).
Pick a finite set of places $S \subset V^k$ as above. Let us first consider the case where $C$ is \emph{split}, i.e. all roots of $f$ are in $k$. Then the Jacobian $J$ of  $C$ has $k$-rational 2-torsion (cf. \cite[Proposition 3]{YM}), so Theorem \ref{T:ExplEst} yields the following.
\begin{prop}\label{P:Examples2}
Let $C$ be a split hyperelliptic curve of genus $g$ over a number field $k$. Then in the above notations, the order of ${}_2\Br(k(C))_V$ divides $2^{a + b + 2g\vert S \vert} \cdot
h_k(S)^{2g}$ where $h_k(S)$ is the class number of the ring $\mathcal{O}_k(S)$ of $S$-integers in $k$; in particular, it divides $2^{(2g+1)\vert S \vert} \cdot h_k(S)^{2g}$.
\end{prop}

In the general case, we let $\ell$ the splitting field of the polynomial $f$. Then the 2-torsion of $J$ is $\ell$-rational. As we have seen in \S\ref{S:PfT2}.2, explicit estimations of the order of ${}_2\Br(k(C))_V$ depend on size of $H^1(\ell/k , {}_2J(\ell))$. Let us consider the ``generic'' situation where the Galois group $\Ga(\ell/k)$ is $S_m$. Since the latter is 2-generated, the order of $Z^1(\ell/k , {}_2J(\ell))$ divides $\vert {}_2J(\ell) \vert^2 = 2^{4g}$. So, the method described  in \S\ref{S:PfT2}.2 shows that the order of ${}_2\Br(k(C))_V$ in the ``generic'' case divides $2^{2g(\vert S^{\ell} \vert + 2)} \cdot h_{\ell}(S^{\ell})^{2g}$, where $h_{\ell}(S^{\ell})$ is the class number of the ring $\mathcal{O}_{\ell}(S^{\ell})$ of $S^{\ell}$-integers in $\ell$.

We note that the torsion subgroups in the Jacobians of hyperelliptic curves over $\Q$ were extensively analyzed by V.P.~Platonov and his collaborators, cf. \cite{Plat}.

\vskip5mm

\section{{\sc Appendix : Finiteness of the unramified Brauer group via \'etale cohomology}}\label{S:Appendix}

In this section, we outline an alternative approach to Theorem 2 that was suggested to us by J.-L. Colliot-Th\'el\`ene \cite{CT}.
For the reader's convenience, we begin by recalling the required input from \'etale cohomology theory. Given a scheme $X$, we will denote by ${\rm Sh}(\Xet)$ the category of abelian sheaves on the \'etale site $\Xet.$ Also, we let $\mathbb{G}_{m, X}$ and $\mu_{n,X}$ (or simply $\mathbb{G}_m$ and $\mu_n$ if there is no confusion) denote, respectively, the multiplicative group scheme and group scheme of $n$th roots of unity on $X$, as well as the corresponding \'etale sheaves.
Finally, we will use the standard notations for the Tate twists of $\mu_n.$ Namely, if $j > 0$, then $\mu_n^{\otimes j} = \mu_n \otimes \cdots \otimes \mu_n$ is the \'etale sheaf given by the tensor product of $j$ copies of $\mu_n.$ We set $\mu_n^0$ to be the constant sheaf $(\Z/ n \Z)_X$. For $j < 0$, we define
$$
\mu_n^{\otimes j} = \Hom (\mu_n^{\otimes(-j)}, (\Z/ n \Z)_X).
$$

The first result is a finiteness theorem for the higher direct images of constructible sheaves.\footnotemark \footnotetext{We recall that for a scheme $X$, a sheaf $\mathfrak{F} \in {\rm Sh}(\Xet)$ is said to be {\it constructible} if each affine open $U \subset X$ has a decomposition into finitely many constructible reduced subschemes $U_i$ of $U$ such that for all $i$, the restriction $\mathfrak{F} \vert_{U_i}$ is locally constant and has finite stalks at all geometric points.} In the statement below, we let $T$ denote a noetherian scheme.

\begin{thm}\label{T-Constructible}
Suppose $T$ is regular of dimension 0 or 1 and $n$ is an integer that is invertible on $T$. Let $X$ and $Y$ be $T$-schemes and $f \colon X \to Y$ be a $T$-morphism of finite type. If $\mathfrak{F} \in {\rm Sh}(X_{{\rm \acute{e}t}})$ is an $n$-torsion constructible sheaf, then the sheaves $R^if_* \mathfrak{F}$ are constructible for all $i \geq 0$.
\end{thm}

\noindent (This is Theorem 1.1 in Deligne's article \cite[``Th\'eor\`emes de finitude"]{Del}; the proof is given in \S\S 2-3 therein.)


\vskip2mm

\noindent We will now use Theorem \ref{T-Constructible} to derive the following statement, which is one of the key ingredients of the cohomological proof of Theorem 2.

\begin{thm}\label{T:CT}
Let $X$ be a smooth integral scheme of finite type over $T = \Spec(A)$,
where $A$ is either a finite field or the ring of $S$-integers in a
number field (with $S$ a finite set of places containing all of the archimedean ones). For any integer $n$ invertible in
$A$, the cohomology groups $\he^i(X , \mu_{n,X})$ are finite
for all $i \geq 0$.
\end{thm}

\begin{proof}(Sketch)
If $A = \mathbb{F}_q$ is a finite field, then according to \cite[Corollary VI.5.5]{Mil}, the groups $H^i(X,\mathfrak{F})$ are finite for all $i$ for any finite, locally constant sheaf $\mathfrak{F}$ whose torsion is prime to char $\mathbb{F}_q$, which yields the statement in this case.


Thus, we can now assume that $A = \mathcal{O}_{K,S}$ is a ring of $S$-integers in some number field $K$,
where $S$ is a finite set of places of $K$ containing $V^K_{\infty}$. Applying Theorem \ref{T-Constructible} to the structure morphism $f \colon X \to T$, we conclude that the higher direct images $\mathrm{R}^q f_*(\mu_{n,X})$ are constructible $n$-torsion sheaves on $T$ for all $q \geq 0.$
Next, fix an algebraic closure $K^{{\rm sep}}$ of $K$, denote by $K_S$ the maximal Galois extension of $K$ contained in $K^{{\rm sep}}$ which is unramified outside $S$, and let $G_{K,S} = \Ga (K_S/K)$. Letting $\bar{\eta} = \Spec (K^{{\rm sep}})$, we recall that the \'etale fundamental group
$\pi_1 (T, \bar{\eta})$ coincides with $G_{K,S}$ (see, e.g., \cite[Corollary 6.17]{Lenstra}), and the functor $\mathfrak{F} \mapsto \mathfrak{F}_{\bar{\eta}}$ defines an equivalence of categories between constructible sheaves on $T$ and finite discrete $G_{K,S}$-modules. Under these identifications, for any prime $\ell$ invertible on $T$, we have an isomorphism of $\ell$-primary parts
$$
H^p (T, \mathfrak{F})(\ell) \simeq H^p (G_{K,S}, \mathfrak{F}_{\bar{\eta}}) (\ell)
$$
for all $p \geq 0$ --- see \cite[Ch. II, Proposition 2.9]{Mil-ADT}. Furthermore, according to \cite[Theorem 8.3.20(i)]{NSW}, if $M$ is a finite $G_{K,S}$-module whose order is invertible in $A$, the groups $H^p (G_{K,S}, M)$ are finite for all $p \geq 0.$ It now follows that $H^p (T, \mathrm{R}^q f_* (\mu_{n,X}))$ are finite for all $p, q \geq 0.$ Finally, the Leray spectral sequence $H^p(T , \mathrm{R}^qf_{*}\mathfrak{F}) \Longrightarrow H^{p+q}(X , \mathfrak{F})$ (\cite[Ch. III, Theorem 1.18]{Mil}) yields the finiteness of
$H^i(X , \mu_{n,X})$ for all $i \geq 0.$

\end{proof}

\noindent {\bf Remark 10.3.} It should be pointed out that the argument sketched above actually shows that for $X$ as in Theorem \ref{T:CT}, the groups $H^i(X, \mathfrak{F})$ are finite for {\it any} $n$-torsion constructible sheaf $\mathfrak{F} \in {\rm Sh}(\Xet).$

\addtocounter{thm}{1}

\vskip3mm

\noindent The second ingredient that will be needed is Grothendieck's absolute purity conjecture, which was proved by O.~Gabber. Let us first recall the relevant notation. Suppose $i \colon Y \hookrightarrow X$ is a closed immersion of schemes and let $\mathfrak{F} \in {\rm Sh}(\Xet)$. Then for any integer $n \geq 0$, one defines sheaves $\mathcal{H}_Y^n (\mathfrak{F}) \in {\rm Sh}(\Yet)$ as follows. Set $U = X \setminus Y$ and let $j \colon U \hookrightarrow X$ be the inclusion. We let
$$
\mathfrak{F}^! = \ker (\mathfrak{F} \to j_*j^* \mathfrak{F}),
$$
where $j_*$ and $j^*$ are the corresponding pushforward and pullback functors, and define $i^! \mathfrak{F} = i^* \mathfrak{F}^!.$ It is easy to see that we obtain in this way a functor
$$
i^! \colon {\rm Sh}(\Xet) \to {\rm Sh}(\Yet)
$$
that is left exact and preserves injectives. Taking right derived functors, we arrive at the definition
$$
\mathcal{H}_Y^n(\mathfrak{F}) = R^n i^! \mathfrak{F}.
$$
In more concrete terms, the correspondence that associates to an \'etale map $f \colon V \to X$ the cohomology group
$H^n_{f^{-1}(Y)} (V, f^* \mathfrak{F})$ is a presheaf on $\Xet$, and one can consider the associated sheaf. It is supported
on $Y$ and coincides with the direct image of the sheaf $\mathcal{H}_Y^n(\mathfrak{F})$.

\begin{thm}\label{T-Purity}{\rm (Absolute purity)}
Let $i \colon Y \hookrightarrow X$ be a closed immersion of noetherian regular schemes of pure codimension $c.$ Let $n$ be an integer that is invertible on $X$. Then for any integer $j$, we have
$$
\mathcal{H}_Y^i (\mu_{n,X}^{\otimes j}) = \left\{ \begin{array}{lr} 0 & {\rm for} \ i \neq 2c \\ \mu_{n, Y}^{\otimes (j-c)} & {\rm for} \ i = 2c \end{array} \right.
$$
\end{thm}

\vskip2mm

\noindent An account of Gabber's first proof, written up Fujiwara, can be found in \cite{Fuj}. A different, more recent approach, also due to Gabber, is explained in \cite{Riou}. Furthermore, we refer the reader to \cite{CT-SB} for an extensive discussion of examples and applications of absolute purity, as well as  \cite[p.153]{CT-S} and \cite[discussion after Theorem 4.2]{CT-Bour} regarding the history of the question.

Next, in view of absolute purity, the local-to-global spectral sequence
$$
E_2^{p,q} = \he^p (Y, \mathcal{H}_Y^q (\mu_{n,X}^{\otimes j})) \Rightarrow H_Y^{p+q} (X, \mu_{n,X}^{\otimes j})
$$
degenerates, leading to isomorphisms
\begin{equation}\label{E-IsomLG}
H_Y^i (X, \mu_{n,X}^{\otimes j}) \simeq H^{i-2c}(Y, \mu_{n,Y}^{\otimes (j-c)}).
\end{equation}
Plugging in (\ref{E-IsomLG}) into the standard long exact sequence
$$
\cdots \to H^i (X, \mu_n^{\otimes j}) \to H^i (U, \mu_n^{\otimes j}) \to H^{i+1}_Y(X, \mu_n^{\otimes j}) \to H^{i+1}(X, \mu_n^{\otimes j}) \to \cdots
$$
yields the so-called Gysin sequence
\begin{equation}\label{E-Gysin}
\cdots \to H^i (X, \mu_n^{\otimes j}) \to H^i (U, \mu_n^{\otimes j}) \to H^{i+1-2c}(Y, \mu_n^{\otimes (j-c)}) \to H^{i+1}(X, \mu_n^{\otimes j}) \to \cdots
\end{equation}

Suppose now that $A$ is a discrete valuation ring with fraction field $K$ and residue field $\kappa.$ We let $X = \Spec(A)$ and set $Y = \Spec(\kappa)$ to be the closed point and $\eta = \Spec(K)$ the generic point. Writing $H^i (A, \mu_n^{\otimes j})$ for $H^i (X, \mu_n^{\otimes j})$ and identifying the \'etale cohomology groups of $Y$ and $\eta$ with the respective Galois cohomology groups for $\kappa$ and $K$, we obtain from (\ref{E-Gysin}) the exact sequence
$$
\cdots \to H^i (A, \mu_n^{\otimes j}) \to H^i (K, \mu_n^{\otimes j}) \stackrel{\partial_A}{\longrightarrow} H^{i-1} (\kappa, \mu_n^{\otimes (j-1)}) \to H^{i+1} (A, \mu_n^{\otimes j}) \to \cdots
$$
for any integer $n$ invertible in $A$. In particular, taking $i = 2$, we obtain a {\it residue map}
\begin{equation}\label{E-ResidueMap}
\partial_A \colon {}_n\Br(K) \to H^1 (\kappa, \Z / n \Z)
\end{equation}
whose kernel is known to coincide with that of the residue map $\rho_v \colon {}_n\Br(K) \to H^1 (\kappa, \Z / n \Z)$ considered in \S\ref{S:Intro}, where $v$ is the discrete valuation of $K$ associated with $A$.

Furthermore, if $X$ is a smooth integral affine scheme of finite type over a finite field or a ring of $S$-integers in a number field with function field $K = K(X)$, then it follows from purity for discrete valuation rings together with the discussion in \cite[\S 3.4]{CT-SB} that there is an exact sequence
\begin{equation}\label{E-BrauerGroupPur}
0 \to {}_n\Br(X) \to H^2 (K, \mu_n) \stackrel{\partial}{\longrightarrow} \bigoplus_{x \in X^{(1)}} H^1(\kappa(x), \Z/ n \Z),
\end{equation}
where $\Br(X) = \he (X, \mathbb{G}_m)$ is the geometric Brauer group, the direct sum is taken over all points $x$ of codimension 1 with residue field $\kappa(x)$, and $\partial$ is the product of the residue maps (\ref{E-ResidueMap}).


\vskip2mm

\noindent We are now in a position to give a sketch of

\vskip1mm

\noindent {\it Cohomological proof of Theorem 2.} Pick a model $X$ for $K$, i.e. a smooth affine integral scheme $X$
as in Theorem \ref{T:CT} with function field $K$. By Theorem \ref{T:CT}, we see that $H^2 (X, \mu_{n,X})$ is finite. The cohomology sequence associated to the Kummer sequence
$$
1 \to \mu_{n,X} \to \mathbb{G}_{m,X} \stackrel{n}{\longrightarrow} \mathbb{G}_{m,X} \to 1
$$
then implies that ${}_n\Br(X)$ is finite. Finally, (\ref{E-BrauerGroupPur}) shows that ${}_n\Br(X)$ is precisely the unramified Brauer group ${}_n\Br(K)_V$, where $V$ is the set of discrete valuations of $K$ associated with the divisors of $X$ (obviously $V$ satisfies conditions (A) and (B)).

\hfill $\Box$

\medskip

{\small {\bf Acknowledgements.} The first-named author was supported by the Canada Research Chair Program and by an NSERC research grant.
The second-named author was partially supported by NSF grant DMS-1301800 and BSF grant 201049. The third-named author was supported by
an NSF Postdoctoral Fellowship. The second and  third authors thankfully acknowledge the hospitality of the IHES (Bures-sur-Yvette) during
the preparation of the final version of this paper.}

\vskip5mm

\bibliographystyle{amsplain}

\end{document}